\documentclass[12pt]{amsart}
\def \R{\mathbb{R}}
\def \Z{\mathbb{Z}}

\def\co{\colon}

\def\id{\mathop\mathrm{id}}

\def\mod{{\mathop\mathrm{mod}\ }}

\usepackage[normalem]{ulem}

\makeatletter
\@namedef{subjclassname@2020}{\textup{2020} Mathematics Subject Classification}
\makeatother

\usepackage[normalem]{ulem}
\usepackage{tikz-cd}

\usepackage{graphicx}
\usepackage{mathabx}
\newtheorem{theorem}{Theorem}[section]
\newtheorem{lemma}[theorem]{Lemma}
\newtheorem{corollary}[theorem]{Corollary}

\newtheorem{problem}[theorem]{Problem}
\newtheorem{proposition}[theorem]{Proposition}
\graphicspath{{./Figures/}}

\theoremstyle{definition}
\newtheorem{definition}[theorem]{Definition}

\newtheorem{remark}[theorem]{Remark}

\usepackage{geometry}

\usepackage{amsthm}

\makeatletter 
\newtheorem*{rep@theorem}{\rep@title}
\newcommand{\newreptheorem}[2]{%
\newenvironment{rep#1}[1]{%
 \def\rep@title{#2 \ref{##1}}%
 \begin{rep@theorem}}%
 {\end{rep@theorem}}}
\makeatother

\newreptheorem{theorem}{Theorem}
\subjclass[2020]{58K30; 58K65, 57R45}
\title[A homotopy invariant of image simple fold maps]{A homotopy invariant of image simple fold maps to oriented surfaces}
\author{Liam Kahmeyer}
\address{Missouri Valley College}
\email{kahmeyerl@moval.edu}
\date{\today}

\author{Rustam Sadykov}
\address{Kansas State University}
\email{sadykov@ksu.edu}

\begin{document}

\begin{abstract} 
	The singular set of a generic map from a closed manifold of dimension at least $2$ to the plane is a smooth closed curve. We study the parity of the number of components of the singular set under the assumption that the map is an image simple fold map, i.e., the map's restriction to its singular set is a smooth embedding.
	
	The image of the singular set of a map to a plane inherits canonical local orientations via so-called chessboard functions. Such a local orientation gives rise to a cumulative winding number, which is an integer or a half integer. When the dimension of the source manifold is even, we also define an invariant $I$ which is the residue class modulo 4 of the sum of twice the number of components of the singular set, the number of cusps, and twice the number of self-intersection points of the image of the singular set. Using the cumulative winding number and the invariant $I$, we show that the parity of the number of connected components of the singular set does not change under homotopy between image simple fold maps provided that one of the following conditions is satisfied: (i) the dimension of the source manifold is even, (ii) the image of the singular set of the homotopy does not have triple self-intersection points, or (iii) the singular set of the homotopy is an orientable manifold with boundary.  
\end{abstract} 
\maketitle

\section{Introduction}

Singular sets of smooth maps $f:M \rightarrow F$ of smooth $n$-manifolds into surfaces played a strong role in recent various discoveries. Studying singular sets of maps, Gay and Kirby~\cite{GK} proved that any smooth closed oriented connected $4$-manifold admits a trisecting map to $\R^2$,  in analogy to the existence of Heegaard splittings for oriented connected closed $3$-manifolds, see also the paper~\cite{BS} by Baykur and Saeki for the existence of a simplified trisection. Kalmar and Stipsicz~\cite{KS} obtained upper bounds on the complexity of the singular set of maps from $3$-manifolds to the plane.  These upper bounds are expressed in terms of certain properties of the link $L \subset S^3$, where the $3$-manifold is obtained via integral surgery along $L$. Ryabichev~\cite{Rya1} gave precise conditions for the existence of maps of surfaces with prescribed loci of singularities. 
Kitazawa ~\cite{Ki} studied simple stable maps (of non-negative dimension) of smooth manifolds to Euclidean target spaces, ($\R^2$, in particular) whose singular sets are concentric spheres.   Saeki~\cite{Saek} and ~\cite{Sae}  showed that every closed connected oriented $3$-manifold admits a stable map to a sphere without definite fold points.  Many $\Z_2$-invariants of stable maps of $3$-manifolds into the plane were found by M.~Yamamoto in \cite{Ya06}. In \cite{Sae19} Saeki constructed an integral invariant of stable maps of oriented closed $3$-manifolds into $\R^2$.

 A generic smooth map is \emph{image simple} if its restriction to the singular set is a topological embedding. 
In the present paper we study under what conditions the numbers $\#|\Sigma(f)|$ and $\#|\Sigma(g)|$ of components of singular sets of two homotopic image simple fold maps $f$ and $g$  of manifolds of dimension $m\ge 2$ to a surface are congruent modulo two. 

This question has been solved in the case $m=2$,  and it is partially answered in the case\ $m=3$.  Namely, 
M.~Yamamoto~\cite{Ya} showed that if $f$ is a map of degree $d$ between oriented closed surfaces of genera $g$ and $h$ respectively, then the parity of $\#|\Sigma(f)|$ is the same as that of $d(h-1)-(g-1)$. On the other hand, in \cite{Sae19} Saeki studied maps of $3$-manifolds into surfaces, and, in particular, gave an example of two image simple fold maps $f, g\co S^3\to \R^2$ such that the parities of the number of components of the singular sets of $f$ and $g$ are different.

Our main result is split into three cases; the first being the case when the source manifold is of even dimension.

\begin{theorem} \label{th:1} Let $f$ and $g$ be two homotopic image simple fold maps 
from a closed manifold $M$ of even dimension  $m\geq2$ to an oriented surface $F$ of finite genus.  %Suppose that $f|_{\Sigma(f)}$ and $g|_{\Sigma(g)}$ are embeddings. 
Then,  the number of components of $\Sigma(f)$ is congruent modulo two to the number of components of $\Sigma(g)$. 
 \end{theorem}

 To prove Theorem~\ref{th:1} we define the cumulative winding number $\omega(f)\in \frac{1}{2}\Z$ for generic maps to parallelized surfaces.  In general,  $\omega$ is not a homotopy invariant. However,  for image simple fold maps $f, g\co M\to F$ to parallelized surfaces, the invariant $\omega$ is integral, and the parities of $\omega(f)$ and $\omega(g)$ agree. Thus, for image simple fold maps, $\omega\in \Z$ is a $\Z_2$-homotopy invariant. We note that the cumulative winding number we introduce in the present paper is different from the rotation numbers considered by Levine~\cite{Le66},   Chess~\cite{Chess}, and Yonebayashi~\cite{Yo97}. 
 
 We now state theorems for the remaining two cases; 
 the first theorem requires that $\Sigma(f)$ does not undergo any $R_3$ moves (see Fig. ~\ref{fig:2}) during homotopy, while the second requires that the singular set of the homotopy is orientable. 
 
 \begin{theorem}\label{th:2}  Let $f$ and $g$ be two homotopic image simple fold maps  $M\to F$, where %with no cusp singular points
\begin{itemize}
\item $M$ is a closed manifold of odd dimension $m> 2$ and $F$ is $\R^2$ or $S^2$, or
\item $M$ is a closed oriented manifold of dimension $3$, and $F$ is an oriented surface.  
 \end{itemize}
 %Suppose that $f|_{\Sigma(f)}$ and $g|_{\Sigma(g)}$ are embeddings. Also, 
Suppose that no $R_3$ moves occur during a generic homotopy from $f$ to $g$. Then, the number of components of $\Sigma(f)$ is congruent modulo two to the number of components of $\Sigma(g)$. 
 \end{theorem}

We note that $R_3$-moves are closely related to triple points of the singular sets $\Sigma(h)$ of maps $h$ to $\R^3$. These are studied by Saeki and T.~Yamamoto~\cite{SY}.

 The proof of Theorem ~\ref{th:2} also utilizes the cumulative winding number $\omega(f)$.  
 % By a theorem of Saeki in ~\cite{Sae}, a closed connected oriented $3$-manifold admits a simple indefinite fold map into $S^2$ if and only if it is a graph manifold.  Thus, the cumulative winding number defines a $\Z_2$-homotopy invariant of graph manifolds. The second theorem in the odd dimensional case imposes the restriction that the singular surface of the homotopy is necessarily oriented. 

  \begin{theorem}\label{th:3} 
Let $f$ and $g$ be two homotopic image simple fold maps  %with no cusp singular points
 from a closed manifold $M$ of dimension  $m\geq2$ to a surface $F$ of finite genus. 
 %Suppose that $f|_{\Sigma(f)}$ and $g|_{\Sigma(g)}$ are embeddings.  Also, 
Suppose the surface $\Sigma(H)$ of singular points of the homotopy $H$ between $f$ and $g$ is orientable. Then, the number of components of $\Sigma(f)$ is congruent modulo two to the number of components of $\Sigma(g)$. 
 \end{theorem}

Let $ \#|A_2(f)|$ be the number of cusps of the map $f$, $\Delta(f)$ the number of self-intersection points of $f(\Sigma)$, and $\#|\Sigma(f)|$  the number of connected components of $f(\Sigma)$. To prove Theorem ~\ref{th:3}, we introduce a modulo $4$  function
\[
   I(f) \equiv \#|A_2(f)| + 2\Delta(f) +  2\#|\Sigma(f)| \ (\mod 4),
\]   
and show that it is invariant under generic homotopy whose singular set is orientable. In particular, the function $I(f)$ is a homotopy invariant provided that the dimension of the manifold $M$ is even and the surface $F$ is orientable by Theorem~\ref{th:5}.
  In ~\cite{Gr},  Gromov introduced and more deeply studied $I(f)$ as an integer-valued function.  

The paper is structured as follows. In section $2$ we review the notions of generic maps, stable maps, and generic families of maps. We note that there are several conflicting definitions of a generic family of maps in the literature and chose one which is the most convenient for the present paper. In section $3$ we review singularities $A_i(f)$ of Morin maps and introduce the manifolds $A_{I}(f) \subset M$ related to multi-singularities of smooth maps.  In section $4$, using the manifolds $A_I(f)$,  we list all moves of singularities which occur under a generic homotopy of maps to $\R^2$. For completeness,  we give a proof that no other moves are possible.  Section $5$ serves to introduce the notion of an abstract singular set diagram.  In section $6$ we define chessboard functions and in section $7$ we look at examples of chessboard functions.  In sections $8$ and $9$ we define the cumulative winding number and record how homotopy affects the cumulative winding number, respectively.  In section $10$ we prove Theorems ~\ref{th:1} and ~\ref{th:2}, and in section $11$, we prove that $I(f)$ is indeed invariant under homotopy with orientable singular set, and use it to prove  Theorem ~\ref{th:3}. We finish our discussion in section $12$ by listing and proving a few applications of our results.

We would like to express our thanks to Osamu Saeki and Masamichi Takase for their comments and references.

\section{Stable and generic maps}\label{s:2}

 In this section we recall the definition of stable maps, generic maps, generic families of maps, and $n$-functions.

Let $f$ be a smooth map of a non-negative dimension $m-n$ of a manifold $M$ of dimension $m$ to a manifold $N$ of dimension $n$. We say that a point $x\in M$ is \emph{regular} if the kernel rank of $f$ at $x$ is $m-n$. Otherwise, the point $x$ is said to be \emph{singular}. 
Recall that a smooth map is a \emph{Thom-Boardman} map if for each $k$,  its $k$-jet extension is transverse to each Thom-Boardman submanifold of the $k$-jet space. 
The singular set $\Sigma(f)$ of a Thom-Boardman map $f\co M\to N$ is stratified by smooth submanifolds $\Sigma^I(f)\subset M$ parametrized by Thom-Boardman symbols $I$. 

\subsection{Generic maps}
Let $f\co M\to N$ be a Thom-Boardman map. Let $x_j\in \Sigma^{I_j}(f)$ be distinct singular points in $M$  with $j=1,..., r$ such that 
\[
    f(x_1)=f(x_2)=\cdots =f(x_r)=y.
\]  
We say that $f$ satisfies the \emph{normal crossing condition} if for each tuple of points $x_1,..., x_r$ as above the vector spaces 
\[
     d_{x_1}f(T_{x_1}\Sigma^{I_1}), ...., d_{x_r}f(T_{x_r}\Sigma^{I_r})
\]
are in general position in the vector space $T_yN$. 

\begin{definition} 
We say that a smooth map $f$ is \emph{generic} if it is a Thom-Boardman map satisfying the normal crossing condition. 
\end{definition}

\noindent It is known that generic maps are residual in $C^{\infty}(M, N)$, e.g. see \cite[p.157]{GG}. 

\subsection{Stable maps}\label{s:stable_maps}
A smooth map $f\co M\to N$ is said to be \emph{stable} if any smooth map $f'$ sufficiently close to $f$ is \emph{right-left equivalent} to $f$, i.e., there are diffeomorphisms $g$ of $N$ and $h$ of $M$ such that $f'=g\circ f\circ h^{-1}$. In fact, there are various equivalent definitions of stability of smooth maps  $f\co M\to N$ of a closed manifold $M$ to an arbitrary manifold $N$, e.g., see  \cite[Chapter V, Theorem 7.1]{GG}.  In particular, $f$ is stable if and only if any $k$-parametric deformation of $f$ is trivial in the sense of \cite[Chapter V, Definition 2.3]{GG}.  

Stable maps are generic, e.g., see \cite[Chapter VI, Theorem 5.2]{GG}. On the other hand, 
it is known that stable maps are not dense in $C^{\infty}(M, N)$, e.g., see \cite[p. 160]{GG}. In particular, the sets of stable maps and generic maps are not the same in general. However, a proper map of a manifold of dimension $m$ to a manifold of dimension $n\le 3$ such that $m\ge n$ is stable if and only if it is generic~\cite{MaV}.

\subsection{Generic families of maps}\label{s:2.3}
Let $f_t\co M\to N$ be a parametric family of maps parametrized by a smooth manifold $T$. It defines a map $F\co M\times T\to N\times T$ by $F(x, t)=(f_t(x), t)$, and a stratification of $M\times T$ by submanifolds $\Sigma^I(F)$, where $I$ ranges over Thom-Boardman symbols. It is common to define a \emph{generic homotopy} $f_t$ by requiring that the associated map $F$ is generic. However, we will need a more restrictive definition. 
Let $\pi_T$ denote the projection of $M\times T$ onto the second factor. 
We say that a parametric family $\{f_t\}$ is a \emph{generic parametric family} if the associated map $F$ is generic, and the restrictions $\pi_T|_{\Sigma^I(F)}$ are generic for each Thom-Boardman symbol $I$. A parametric family $\{f_t\}$ is a \emph{stable parametric family} if any $k$-parametric deformation of $F(x, t)=(f_t(x), t)$ is trivial.

\subsection{$n$-functions}

In some cases it is helpful to study maps to manifolds of dimension $n$ by means of $(n-1)$-parametric families of functions, or, $n$-functions.  More precisely, given a manifold $X$ of dimension $m$, and a manifold $Y$ of dimension $n\le m$, a smooth proper map $f\co X\to Y$ is an \emph{$n$-function} if for each $q\in Y$,  there is a compact neighborhood $U$ of $q$ with a diffeomorphism $\psi\co U\to [0,1]^n$,  and a diffeomorphism $\varphi\co f^{-1}(U)\to [0,1]^{n-1}\times M$ for an $(m-n+1)$-manifold $M$,  such that $\psi\circ f\circ \varphi^{-1}: [0,1]^{n-1}\times M\to [0,1]^{n-1}\times [0,1]$ is of the form $(t, p)\mapsto (t, g_t(p))$,  for some parametric family $g_t$ of functions on $M$.  A generic $2$-function is also called a Morse $2$-function, see \cite[Definition 2.7]{GK}.  

\begin{lemma}\label{l:8} Let $f\co X\to Y$ be a generic smooth proper map of corank $1$ to a manifold of dimension $n$. Then $f$ is an $n$-function. 
\end{lemma}
\begin{proof} Let $q$ be a point in $Y$. Since $f$ is of corank $1$, there is a diffeomorphism $\psi\co U\to [0,1]^n$ of a neighborhood $U$ of $q$ such that the composition $\pi^\perp_n\circ \psi \circ f|_{f^{-1}(U)}$ is a submersion, where $\pi^\perp_n\co [0,1]^n\to [0,1]^{n-1}$ is the projection $(x_1,..., x_n)\mapsto (x_1,..., x_{n-1})$. We may choose $U$ so that the resulting proper submersion to a disc is a trivial fiber bundle. Then,  there is a diffeomorphism $\varphi\co f^{-1}(U)\to [0,1]^{n-1}\times M$ such that the map $\psi\circ f\circ \varphi^{-1}$ is of the form $(t, p)\to (t, g_t(p))$, for a parametric family $g_t$ of functions on a manifold $M$ of dimension ${m-n+1}$. 
\end{proof}

\section{Singularities of maps}\label{s:3}
 In this section we review the definition of generic singularities of smooth maps to surfaces and manifolds of dimension $3$.

 Let $f$ be a smooth map $f\co M\to N$ of non-negative dimension $m-n$ of a manifold $M$ of dimension $m$ to a manifold $N$ of dimension $n$. 
The set $A_0(f)$ of regular points of $f$ is an open submanifold of $M$ of codimension $0$. We now review the definition of singularity types $A_r$ for $r\ge 1$ with Thom-Boardman symbol $I_r=(m-n+1, 1, ..., 1, 0)$ of length $r+1$. 

We say that a point $x\in M$ is a \emph{fold point} if 
there is a neighborhood $U\cong \R^{n-1} \times \R^{m-n+1}$ about $x$,  with coordinates $(x_1,..., x_m)$ in $M$,  and a coordinate neighborhood $V\cong \R^{n-1}\times \R$ about $f(x)$ in $N$ such that $f(U)\subset V$ and $f|_U$ is given by a product of the identity map $\id_{\R^{n-1}}\co \R^{n-1}\to \R^{n-1}$ and a standard Morse function $\R^{m-n+1}\to \R$ with a unique critical point,  i.e., 
\begin{equation}\label{eq:mor1}
f(x_1,x_2, ..., x_m)=(x_1, ..., x_{n-1}, \pm x_n^2 \pm x_{n+1}^2\pm ... \pm x_{m}^2).
\end{equation}
 The set of fold  points of $f$ is denoted by $A_1(f)$. The number $i$ of terms in (\ref{eq:mor1}) among $x_n, ..., x_m$ with negative signs is called a \emph{relative index} of $f$.  
We may always choose coordinate neighborhoods so that $i\le m-n+1-i$. The number $i$ with respect to such a coordinate system is said to be the (absolute) \emph{index} of the fold point.  If the index of the critical point is $0$,  then  $x$ is said to be a \emph{definite} fold point. Otherwise, the fold point $x$ is \emph{indefinite}. 
\begin{definition}
We say that the map $f$ is a \emph{fold map} if every singular point $x$ is a fold point. Furthermore, a fold map $f$ is an \emph{indefinite fold map} if every fold point is indefinite. 
\end{definition}
\noindent It immediately follows that if $f$ is a fold map, then the set of singular points $\Sigma(f)$ of $f$ is a closed submanifold of $M$ of dimension $n-1$, and $f|_{\Sigma(f)}$ is an immersion. 

We say that a point $x\in M$ is an $A_r$-singular point for $r>1$, if there is a neighborhood $U\subset M$ of $x$,  with coordinates $(t_1,..., t_{n-r}, \ell_2, ..., \ell_r, x_1,..., x_{m-n+1})$,  and a neighborhood $V\subset N$ of $f(x)$,  with coordinates $(T_1,..., T_{n-r}, L_2,..., L_r, Z)$,  such that $f(U)\subset V$ and the restriction $f|_U$ is given by 
\[
    T_i=t_i \qquad \mathrm{for}\quad i=1, ..., n-r, 
\]
\[
   L_i=\ell_i \qquad \mathrm{for}\quad i=2, ..., r, 
\]
\[   
     Z= \pm x_1^2 \pm x_2^2 \pm \cdots \pm x_{m-n}^2 + \ell_2 x_{m-n+1} +  \ell_3 x_{m-n+1}^2 + \cdots + \ell_r x^{r-1}_{m-n+1} \pm  x^{r+1}_{m-n+1}. 
\]
Given a map $f$, the sets $A_r(f)$ of its singular points of type $A_r$  are submanifolds of $M$ of dimension $n-r$. 

\begin{definition}
Singular points of types $A_2$ and $A_3$ are called \emph{cusp} and \emph{swallowtail} singular points, respectively. 
\end{definition}

\begin{definition} 
We say that a stable map $f:M \rightarrow N$ of a smooth manifold $M$ into a surface $N$ is \emph{simple} if $A_2(f) = \emptyset$,  and for every singular value $y$, every connected component of the singular fiber $f^{-1}(y)$ contains at most one singular point.  Also,  a generic map $f\co M\to N$ is said to be \emph{image simple} if its restriction to the singular set $f|_{\Sigma(f)}$ is a topological embedding.
\end{definition}

We note that the term 'image simple map' is introduced by Saeki in his forthcoming paper. 

\subsection{Morin Maps}
We say that a smooth map $f$ is a \emph{Morin map} if all its singular points are of type $A_r$ for $r\ge 1$. It is known that for $n\le 3$,  all generic maps $M^m\to \R^n$ of non-negative dimension $m-n$ are Morin. The singular set $\Sigma(f)$ of a Morin map is a closed smooth submanifold of $M$ of dimension $n-1$. Given a Morin map $f$, for each $i$, the closure $\mathop\mathrm{Cl}(A_i(f))$ is a smooth submanifold of $M$ possibly with boundary. Furthermore, for each $i$ and $j$ such that $i<j$, the manifold $\mathop\mathrm{Cl}(A_j(f))$ is a submanifold of $\mathop\mathrm{Cl}(A_i(f))$. For a generic Morin map $f$, we denote by $A_{ij}(f)$ the set of points $x\in A_i(f)$ for which there is a distinct point $y\in A_j(f)$ such that $f(x)=f(y)$. Similarly, we denote by $A_{ijk}(f)$ the subset of points $x\in A_i$ for which there are distinct points $y\in A_j(f)$ and $z\in A_k(f)$ such that $f(x)=f(y)=f(z)$.  We will denote the restriction of $f$ to $A_{I}(f)$ by $f|_{A_I}$,  for short, where $I$ is either a single index $i$, or a multi-index $ij$ or $ijk$. 

When the non-negative dimension $m-n$ of a map $f:M \rightarrow N$ is even, we have the following theorem. 

\begin{theorem}\label{th:5}  Let $f\co M \to N$ be a Morin map of non-negative even dimension $m-n$ into an oriented manifold $N$. Then,  the set $\Sigma(f)$ is a canonically oriented submanifold of $M$. 
\end{theorem}

We emphasize that the manifold $M$ in Theorem~\ref{th:5} is not necessarily orientable.

\begin{proof} Since $\mathop\mathrm{Cl}(A_3(f))$ is a proper submanifold of $\Sigma(f)$ of codimension $2$, the manifold $\Sigma(f)$ is orientable if and only if the manifold $\Sigma(f)\setminus \mathop\mathrm{Cl}(A_3(f))$ is orientable. Thus, to prove Theorem~\ref{th:5}, 
it suffices to introduce an orientation of $\Sigma(f)$ in the complement to $\mathop\mathrm{Cl}(A_3(f))$, i.e., only over the union of $A_1(f)$ and $A_2(f)$. 

We note that the index of a fold point $x$ depends on the choice of coorientation of the immersed submanifold $f(A_1)$ at the point $f(x)$. If a fold point $x$ is of index $i$ for one choice of coorientation, then $x$ is of index $m-i-n+1$ for the other choice of  coorientation. Since the (non-negative) dimension $m-n$ of the map $f$ is even, it follows that the parity of the index is changed when the coorientation is changed. Consequently, the immersed manifold of fold points $f(A_1)$ admits a unique coorientation at each point $f(x)$,  such that the index of the fold point $x$ with respect to the coorientation is odd. We say that such a coorientation is \emph{canonical}. 

We orient the immersed manifold $f(A_1)$ so that the orientation of $f(A_1)$ followed by the coorientation of $f(A_1)$ agrees with the orientation of $N$. In turn, the orientation of $f(A_1)$ defines an orientation of $A_1(f)$. 
We claim that the so-defined orientation of $A_1(f)$ extends to an orientation of $A_1(f)\cup A_2(f)$. Indeed, in a coordinate neighborhood $U$ about an $A_2$-singular point $x$ and a coordinate neighborhood about $f(x)$,  the map $f$ is given by:
\[
  T_i = t_i, \qquad i=1, ..., n-2, 
  \]
  \[
  L_2=\ell_2, 
  \]
  \[
  Z= \varphi_{\ell_2} (x_{m-n+1}) \pm x_1^2 \pm x_2^2 \pm \cdots \pm x_{m-n}^2, 
  \]
  where $(t_1,.., t_{n-2}, l_2, x_1,..., x_{m-n}, x_{m-n+1})$ are local coordinates about $x$ in $M$ and $(T_1, T_2, ..., T_{n-2}, L_2, Z)$ are local coordinates about $f(x)$ in $N$, and for each $l_2$, 
  \[
\varphi_{\ell_2} (x_{m-n+1})=  \ell_2 x_{m-n+1} + x_{m-n+1}^3
\] 
is either a regular function, a Morse function with a cancelling pair of critical points, or a function with a unique birth-death singularity. For each fold critical value in $f(\Sigma\cap U)$, the direction $\frac{\partial}{\partial Z}$ defines a coorientation of $f(\Sigma)$ at the corresponding critical value. It follows that for each Morse function $\varphi_{(t_1,..., t_{n-2}, \ell_2)}$,  the parities of the indices of the two cancelling Morse critical points are different. 
Therefore, the canonical coorientation of one critical point of $\varphi_{(t_1,..., t_{n-2}, \ell_2)}$ is given by $\frac{\partial}{\partial Z}$, while the canonical coorientation for the other critical point is $-\frac{\partial}{\partial Z}$.  Thus, the coorientation of $f(A_1)$ extends to a coorientation of an immersed smoothing of $f(A_1\cup A_2)$. This implies that $\Sigma(f)$ is orientable. 
\end{proof}

\subsection{Singularities of generic maps to $2$-manifolds}
Let $f\co M\to N$ be a generic smooth map of a manifold of dimension $m\ge 2$ to a manifold $N$ of dimension $2$. The map $f$ may only have regular, fold, and cusp points. 
The set of regular points forms an open submanifold $A_0(f)$ of $M$. The complement to the submanifold $A_0(f)$ in $M$ is the submanifold of singular points $\Sigma(f)$ of dimension $1$.  It contains a discrete set of cusp singular points $A_2(f)$. The rest of $\Sigma(f)$ is a disjoint union of arcs and circles of fold points $A_1(f)$. 
The restriction of $f$ to $A_0(f)$ is a submersion. The restriction of $f$ to $A_1(f)$ 
is a self-transverse immersion with $0$-dimensional self-crossings.  The images of $f|_{A_1}$ and $f|_{A_2}$ are disjoint. 

\subsection{Singularities of generic maps to $3$-manifolds}
Let $F\co M\to N$ be a generic smooth map of a manifold of dimension $m\ge 3$ to a manifold of dimension $3$.  The map $F$ may only have regular, fold, cusp, and swallowtail map germs. 
Since $F$ satisfies the normal crossing condition, the set $A_{11}(F)$ is a submanifold which consists of open arcs and circles. We note that the image of $A_{11}(F)$ is the self-crossing of the immersion $F|_{A_1}$, while the image of $A_{12}(F)\cong A_{21}(F)$ is the set of intersections of folds with cusps. The image of the set $A_{111}(F)$ is the set of triple self-intersections of folds.  The submanifolds $A_{12}(F)\subset A_1(F)$, 
$A_{21}(F)\subset A_2(F)$ and $A_{111}(F)$ are of dimension $0$, while all other manifolds $A_{ij}(F)$ and $A_{ijk}(F)$ (except for the aforementioned manifold $A_{11}(F)$) are empty. 

\section{Generic homotopies of maps to $\R^2$}\label{s:4}
 
 In this section we study how the singular set of a map to $\R^2$ is modified under generic homotopy, i.e., under generic one parameter family of maps.

Let $F\co M\times [0,1]\to \R^2\times [0,1]$ be a homotopy between two generic maps of a closed manifold $M$, and let $\pi\co M\times [0,1]\to [0,1]$ denote the projection onto the second factor. 
Then the homotopy $F$ is a \emph{generic homotopy} if $F$ is a generic map and $\pi|_{A_I(F)}$ is a generic function for each $I\in \{\{1\}, \{2\},\{11\}\}$, see \S\ref{s:2.3}.

\begin{lemma}\label{lem:3} The set of generic homotopies is open and dense in the space of all homotopies. 
\end{lemma}
\begin{proof}  
Any homotopy $F'$ sufficiently close to a generic homotopy $F$ is also generic. Indeed, a map to a manifold of dimension $3$ is generic if and only if it is stable. In particular, a generic homotopy $F$ is a stable map, see \S\ref{s:stable_maps}. Hence any homotopy $F'$ sufficiently close to $F$ is right-left equivalent to $F$, and in particular, generic. 
 Consequently, the set of generic homotopies is open.  

 Next, let us show that the set of generic homotopies is dense.
By the Mather theorem~\cite{MaVI}, stable maps to manifolds of dimension $3$ are dense. Consequently, any homotopy $F$ can be approximated by a stable map $F'\co (M; M\times \{0\}, M\times \{1\})\to ([0,1], 0, 1)$. On the other hand, an approximation of a homotopy is a homotopy. Thus, every homotopy $F$ can be approximated by a homotopy that is a stable map.  In particular, we may assume that $F$ is itself a generic map. 

If $F$ is a generic map, then $\overline{A_1(F)}$ is a closed surface, while $\overline{A_2(F)}$ and $\overline{A_{11}(F)}$ are closed curves. There is a diffeomorphism $\varphi \in C^\infty(M\times [0,1], M\times [0,1])$ close to the identity map of $M \times [0,1]$ such that the restrictions of $\pi\co M\times [0,1]\to [0,1]$ to $A_I(F\circ \varphi)$ are generic functions for each $I\in \{\{1\}, \{2\}, \{1,1\}\}$.   If $\varphi$ is chosen sufficiently close to the identity map of $M \times [0,1]$,    then $F\circ \varphi$ is a generic homotopy close to $F$. Thus, every neighborhood of $F$ contains a generic homotopy, i.e., the set of generic homotopies is dense. This completes the proof of Lemma~\ref{lem:3}. 
\end{proof}

We note that members $f_t$ of a generic family $F=\{f_t\}$ may not be generic maps. We will next list several instances when a member $f_t$ of a generic homotopy of maps to $\R^2$ is not generic. This list is exhaustive when $F$ is a generic homotopy, see Theorem~\ref{th:4}. 

\subsection{List of generic moves}

\subsubsection{Reidemeister-II fold crossing}
The restriction $f_t|_{A_1}$ may not be a self-transverse immersion for a discrete set of moments $t \in [0,1]$. If $f_t$ is a generic homotopy, and $f_{t}|_{A_1}$ is not self-transverse at $t=t_0$, then as $t$ ranges in the interval $(t_0-\varepsilon, t+\varepsilon)$, the map $f_{t}|_{A_1}$ undergoes a Reidemeister-II fold crossing, see Fig. ~\ref{fig:1}.

\begin{figure}
\includegraphics[width=0.6\textwidth]{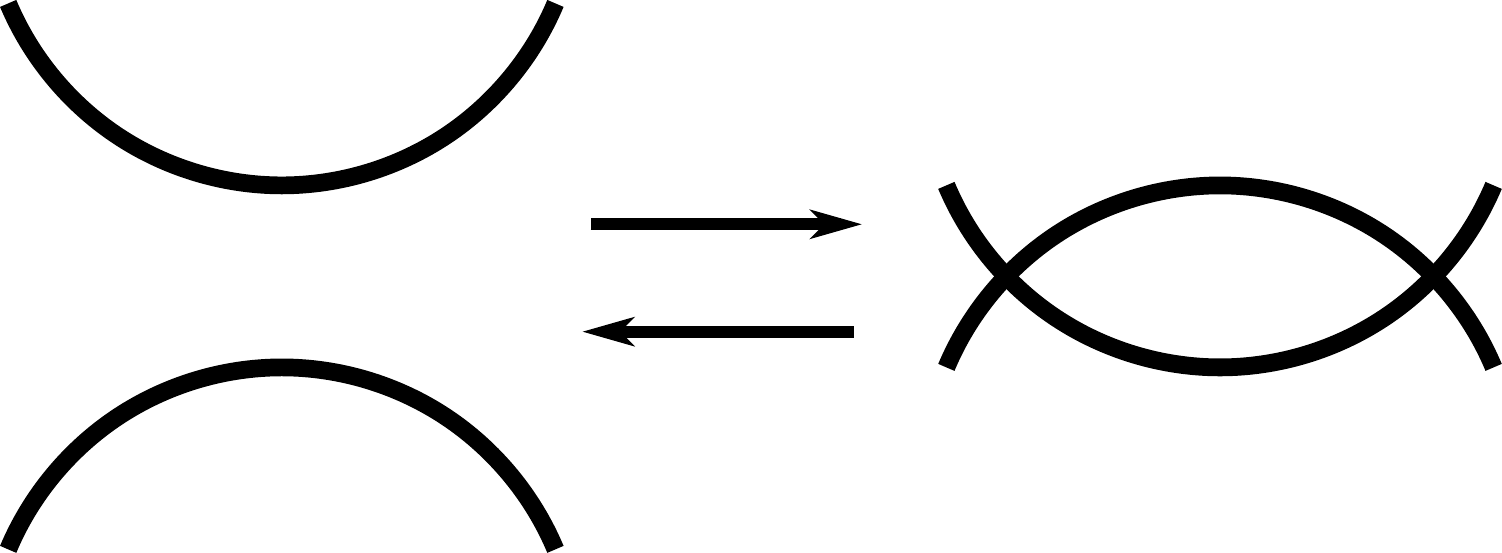}
\caption{Reidemeister-II fold crossing}
\label{fig:1}
\end{figure}

\begin{figure}
\includegraphics[width=0.6\textwidth]{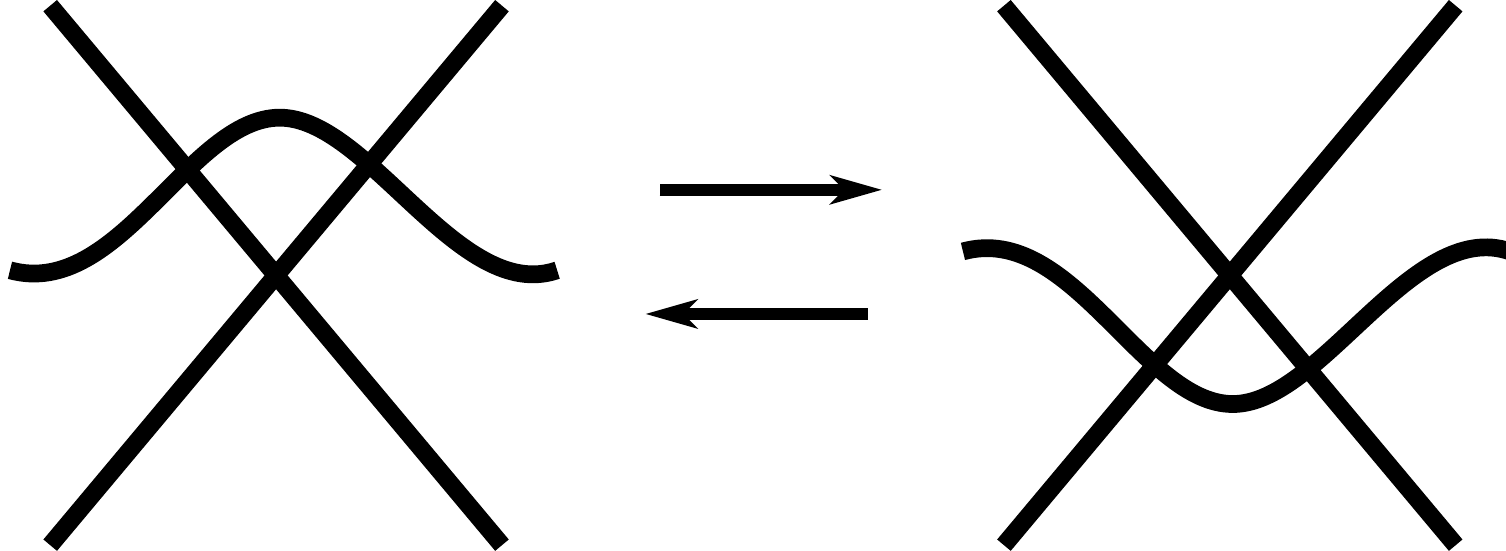}
\caption{Reidemeister-III fold crossing}
\label{fig:2}
\end{figure}

\subsubsection{Reidemeister-III fold crossing}  Similarly, the map $f_t|_{A_1}$ may undergo a Reide\-meister-III fold crossing, see Fig.~\ref{fig:2} 
\begin{figure}
\includegraphics[width=0.6\textwidth]{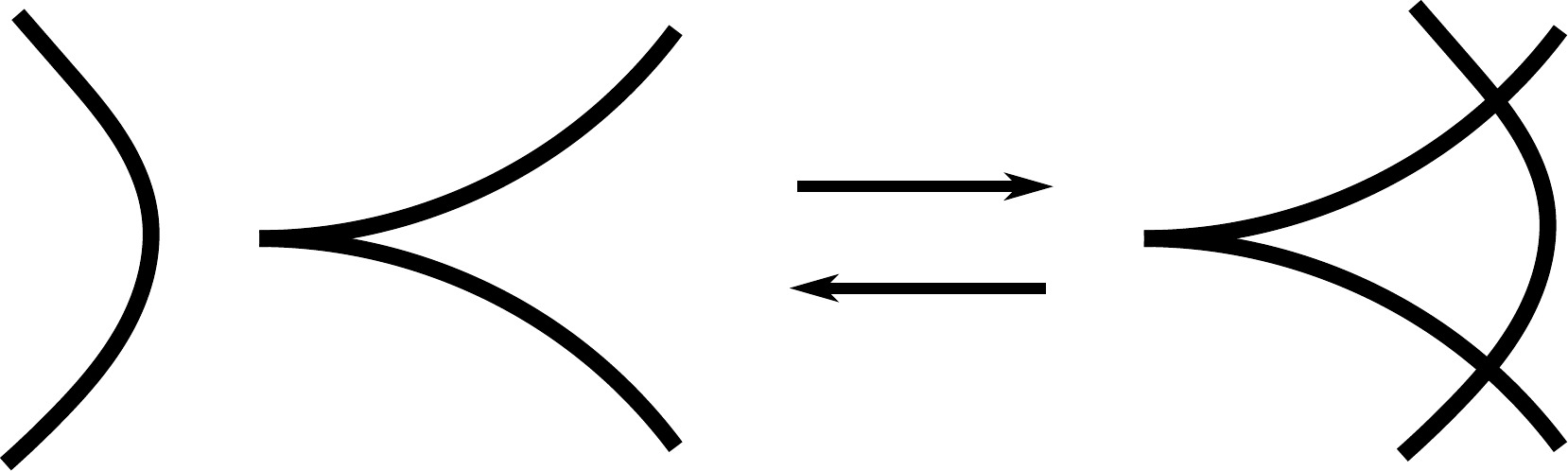}
\caption{A cusp passing through a fold curve}
\label{fig:3}
\end{figure}

\subsubsection{Cusp-fold crossing}  The cusp-fold crossing occurs when $f_t(x)=f_t(y)$,  for a cusp point $x\in A_2(f_t)$ and a fold point $y\in A_1(f_t)$, see Fig.~\ref{fig:3}. 

\begin{figure}
\includegraphics[width=0.6\textwidth]{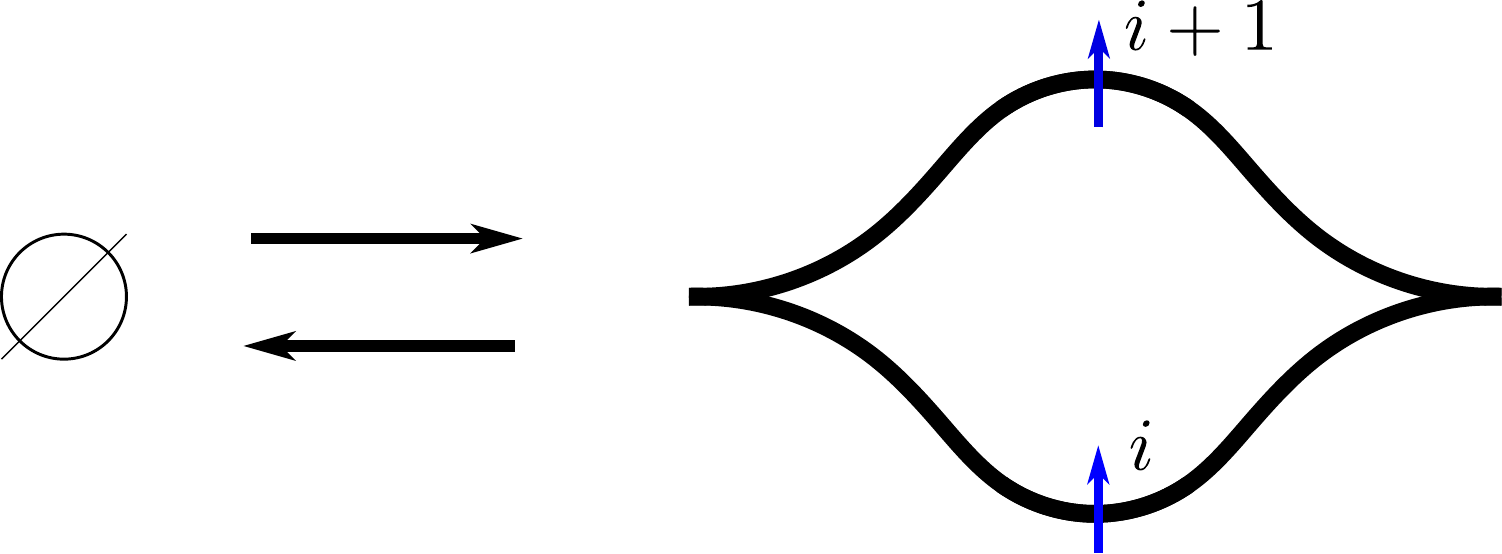}
\caption{Wrinkle singularity}
\label{fig:4}
\end{figure}

\break
In Figures ~\ref{fig:4}, ~\ref{fig:5}, and ~\ref{fig:6}, the numbers $i$ and $i+1$ indicate the relative index of each fold curve. The relative index for each curve is considered in the direction of the corresponding blue arrow.

\subsubsection{Wrinkle singularity} Under a generic homotopy,  a new path component of singular points may appear in the form of a wrinkle, see Fig.~\ref{fig:4}. 

\begin{figure}
\includegraphics[width=0.6\textwidth]{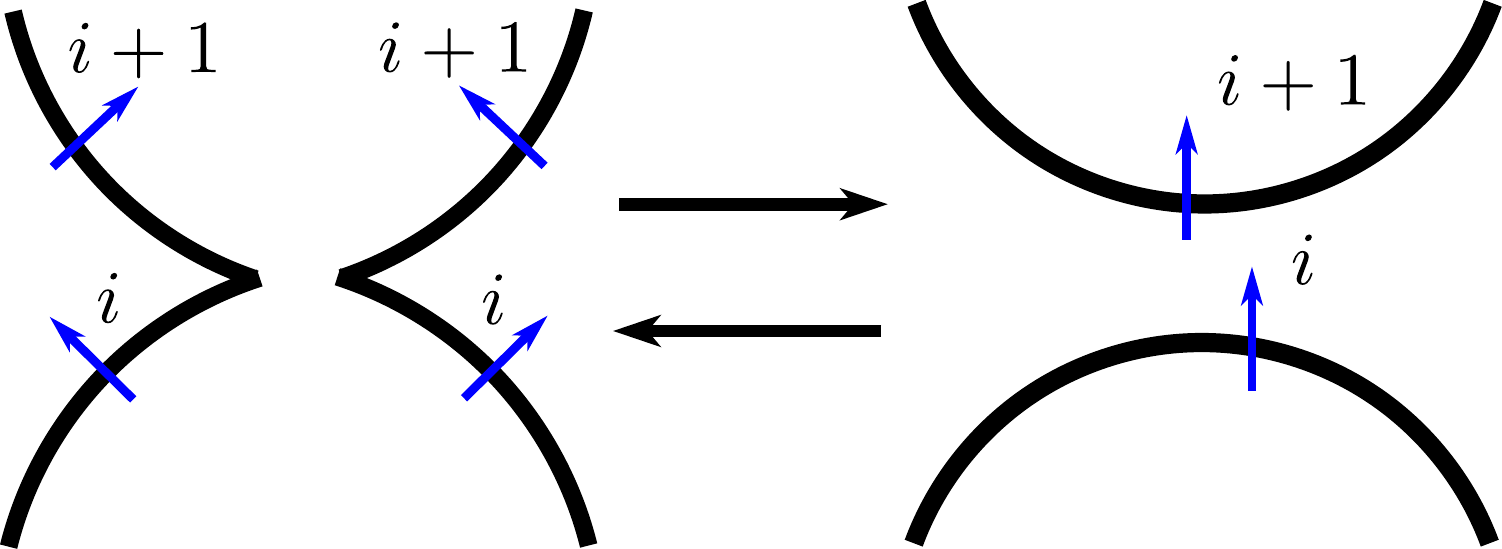}
\caption{Merging and unmerging 2 cusps}
\label{fig:5}
\end{figure}

\subsubsection{Merge singularity} Under a merge singularity move,  a canceling pair of cusp points disappear while the singular set changes by a surgery of index $1$ along the canceling pair of cusp points,  see Fig.~\ref{fig:5}. 

\begin{figure}
\includegraphics[width=0.6\textwidth]{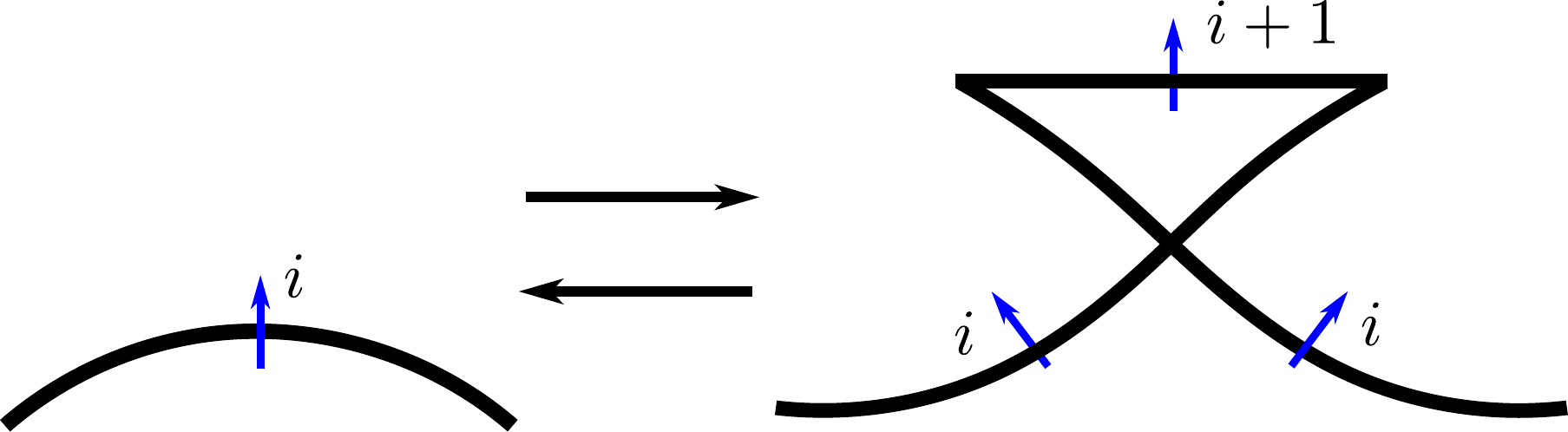}
\caption{Introduction of a swallowtail}
\label{fig:6}
\end{figure}

\subsubsection{Swallowtail singularity} Under a swallowtail singularity move,  two cusp points and a self-intersection point of the singular set appear,  see Fig.~\ref{fig:6}. 

\begin{theorem}\label{th:4}  Under a generic homotopy $F = \{f_t\}$ of maps to $\R^2$, the singular set $\Sigma(F)$ is modified by isotopy,  as well as the above listed moves.
\end{theorem}
\begin{proof}  Let $F\co M\times [0,1]\to \R^2\times [0,1]$ be a generic homotopy, and let $\pi\co M\times [0,1]\to [0,1]$ denote the projection to the second factor. If $\pi|_{A_I(F)}$ does not have critical points on the level $M\times \{t_0\}$, then for sufficiently small $\varepsilon>0$, the singular set $A_I(f_t)$,  parametrized by $t\in [t_0-\varepsilon, t_0+\varepsilon]$,  is modified by an ambient isotopy.  Thus, it remains to study modifications of the singular set of $f_t$ corresponding to critical points of the generic functions $\pi|_{A_I(F)}$.  We claim that $\pi|_{A_I(F)}$ has no critical points when $I=\{1\}$. 

\begin{lemma}\label{le:4}  The map $\pi|_{A_1(F)}$ is a submersion.
\end{lemma}
\begin{proof}   
Over the set $A_1(F)$ of critical points, there is a well-defined kernel bundle $K_1(F)$ of $dF$. In fact, over $A_1(F)$ there is a splitting  
\[
       T(M\times [0,1])|_{A_1(F)} \cong K_1(F)|_{A_1(F)} \oplus TA_1(F). 
\]
Assume that there is a critical point $p\in A_1(F)$ of the function $\pi|_{A_1(F)}$. Then $T_p(A_1(F))$ is in the kernel of $d_p\pi$. On the other hand, the projection $d_p\pi$ coincides with the composition 
\[
      T_p(M\times [0,1])\longrightarrow T_{F(p)}(\R^2\times [0,1])\longrightarrow T_{\pi(p)}([0,1])
\]
of $d_pF$ and the differential of the projection $\R^2\times [0,1]\to [0,1]$ onto the second factor. Since $K_1(F)|_p$ is in the kernel of $d_pF$, it follows that $K_1(F)|_p$ is in the kernel of $d_p\pi$. To summarize, we have shown that $T_p(M\times [0,1])$ is in the kernel of $d_p\pi$, which contradicts the fact that $\pi$ is a submersion.  
\end{proof}

\noindent Let us now consider critical points of the function $\pi|_{A_2(F)}$. 

\begin{lemma}\label{le:5} Let $p\in A_2(F)$ be a critical point of $\pi|_{A_2(F)}$. Then $p$ is a critical point of $\pi|_{\Sigma(F)}$. 
\end{lemma}
\begin{proof}  As above, over $A_2(F)$, there is a well-defined kernel bundle $K_1(F)$. 
\iffalse
which fit the exact sequence
 \[
      0\longrightarrow K_1(F) \longrightarrow T(M\times [0,1])|_{A_2(F)}\longrightarrow F^*T(\R^2\times [0,1])|_{A_2(F)}
      \longrightarrow Q_1(F) \longrightarrow 0. 
\]
\fi
Let $L$ denote the vector subbundle of $T(M\times [0,1])|_{A_2(F)}$ given by $K_1(F)\cap T(\Sigma(F))$. It follows that $\dim L=1$, and there is a splitting
\[
     T(\Sigma(F))|_{A_2(F)} \cong L \oplus T(A_2(F)). 
\]
Since $L_p$ belongs to the kernel $K_1(F)|_p$ of $d_p\pi$, it belongs to the kernel of $d_p\pi|_{\Sigma(F)}$.
On the other hand, if $p$ is a critical point of $\pi|_{A_2(F)}$, then $T_p(A_2(F))$ is also in the kernel of $d_p\pi|_{\Sigma(F)}$. Thus, the point $p$ is a critical point of $\pi|_{\Sigma(F)}$. 
\end{proof}

By Lemma~\ref{le:5}, if $p$ is a critical point of $\pi|_{A_2(F)}$, then $p$ is also a critical point of the function $\pi|_{\Sigma(F)}$. If the index of the critical point $p$ is $0$, then $p$ corresponds to the appearance (birth) of a wrinkle singularity in $\Sigma(f_t)$. A critical point of index $1$ corresponds to the cusp merge move or its inverse, while a critical point of index $2$ corresponds to the disappearance (death) of a wrinkle singularity. 

The critical points of $\pi$ restricted to the submanifold of double points of $A_{11}(F)$ correspond to Reidemeister-II fold crossings. All points of $A_{12}(F), A_{111}(F)$,  and $A_{3}(F)$ are critical in the sense that the differential of $\pi|_{A_I(F)}$ in these cases vanishes. 
It remains to observe that points of $A_{12}(F)$ correspond to cusp-fold crossings,  $A_{111}(F)$ correspond to  Reidemeister-III fold crossings, and $A_3(F)$ correspond to swallowtail singularities.  
\end{proof}

\begin{remark}  
 The counterpart of Lemma~\ref{le:4} for a generic concordance $F\co M\times [0,1]\to \R^2\times [0,1]$ of smooth maps is not valid. There are moves of generic concordances that do not occur under a generic homotopy. Specifically, under a generic concordance,  an embedded circle of fold points may appear or disappear,  and the curves of fold points may be modified by embedded surgery of index $1$. 
\end{remark}

\section{Oriented abstract singular set diagrams}

The proof of the main results relies on so-called abstract singular set diagrams, which we introduce now.

Let $S$ denote a closed (possibly not path-connected) manifold of dimension $1$  together with two disjoint families $P\subset S$ and $Q\subset S$,  of finitely many distinguished points. We require that the number of points in $Q$ is even, and that the points in $Q$ are paired.  We denote the distinguished points in the family $P$ by $p_1, p_2, ...$, and the points in $Q$ by $q_1, q_1', q_2, q_2', ...$, where the points $q_i$ and $q_i'$ are paired.  We say that a compact subset of $S$ is an \emph{arc} if its interior contains no distinguished points, and its boundary is either empty or consists of the distinguished points.  
\begin{definition}
An \emph{oriented abstract singular set diagram} consists of the manifold $S$, the families $P$ and $Q$, and an orientation of all arcs on $S$ such that 
\begin{itemize}
\item if two arcs $\alpha$ and $\beta$ share a common point $p_i\in P$, then the orientations of $\alpha$ and $\beta$ agree
\item if $q_j \in Q$ is a common point of arcs $\alpha$ and $\beta$, while $q_j' \in Q$ is a common point of arcs $\alpha'$ and $\beta'$, then the orientations on $\alpha$ and $\beta$ agree if and only if the orientations on $\alpha'$ and $\beta'$ agree. 
\end{itemize}
\end{definition}

In the stated requirements,  we allow that some of the arcs $\alpha, \beta, \alpha'$ and $\beta'$ may coincide.  We note that as a point $x$ traverses a path component of $S$, the orientation of $S$ at $x$,  that agrees with the orientation of an arc containing $x$,  may change only at a point in $Q$. Furthermore, at a point in $Q$ the orientation of $S$ may or may not change.  For the sake of convenience, we will simply refer to an oriented abstract singular set diagram as a \emph{diagram}. 

\section{Chessboard functions}

In order to properly equip a singular set diagram with a so-called canonical local orientation and coorientation, we first need to introduce the concept of a chessboard function. Let $f\co M\to N$ be a generic smooth map of a closed manifold $M$ of dimension $m$ to an oriented manifold $N$ of dimension $n$. We say that a curve $\gamma$ in $N$ is a \emph{generic curve} with respect to $f(\Sigma)$ if it intersects each Thom-Boardman stratum $f(\Sigma^I)$ of the singular set transversely.  In particular, we have $\gamma\cap f(\Sigma)=\gamma\cap f(\Sigma^{d+1,0})$,  where $d=m-n$ is the dimension of the map $f$. 

\iffalse{\color{red}\sout{ To justify the definition, we note that any curve arbitrarily close to a generic curve is generic, and any curve in $\R^n$ can be approximated by generic curves. Indeed, a generic map $f$ has only finitely many different non-empty Thom-Boardman singularities $\Sigma^I(f)$, and each Thom-Boardman singular set $f(\Sigma^I)$ is an immersed submanifold in $\R^n$. Therefore, any curve $\gamma$ can be perturbed slightly,  so that it is transverse to each Thom-Boardman immersed manifold $f(\Sigma^I)$. We claim that the perturbed curve $\gamma$ intersects the singular set $f(\Sigma)$ only at fold singular points.  Indeed, recall that the codimension of the singular submanifold $\Sigma^{i}(f) \subset M$ is $i(n-m+i)$.  Therefore, the codimension of $f(\Sigma^{i}) \subset N$ is $(n-m)+i(n-m+i)$, and, in particular for $i=d+2$, we have }
\[
  (n-m)+i(n-m+i) = -d + 2(d+2)=d+4\ge 4.
\]
\sout{Consequently, $\dim N-  \dim f(\Sigma^I)\ge 4$,  for $I\ge (d+2, 0)$, and $\dim N-  \dim f(\Sigma^I)\ge 2$,  for $I\ge (d+1, 1, 0)$. Thus, a generic curve $\gamma \subset N$ intersects the singular set $f(\Sigma)$ only at fold singular points. }}
 \fi 
 
If necessary,  we can further perturb the generic curve $\gamma$, so that it avoids self-intersection points of the immersed fold surface $f(A_1)$. 
\begin{definition}
We say that a locally constant function $c: N\setminus f(\Sigma) \rightarrow \Z$ (respectively, $c: N\setminus f(\Sigma) \rightarrow \Z_2$) is an \emph{integral chessboard function} (respectively, a \emph{$\Z_2$-valued chessboard function}) if the values $c(\gamma(-1))$ and $c(\gamma(1))$ differ by precisely $1$  for each generic curve $\gamma\co [-1, 1]\to N$ intersecting $f(\Sigma)$ at a unique point $\gamma(0)$. 
\end{definition}

We say that a singular value $y$ of a map $f$ is a  \emph{simple singular value} if the fiber $f^{-1}(y)$ contains a unique singular point. We note that for a generic smooth map,  the submanifold of $N$ of simple fold values is dense in $f(\Sigma)$. A \emph{local orientation} of $f(\Sigma)$ is an orientation of the submanifold of simple fold values.  Similarly, a \emph{local coorientation} of $f(\Sigma)$ is a coorientation in $N$ of the submanifold of simple fold values. We say that a local orientation of $f(\Sigma)$ agrees with the local coorientation of $f(\Sigma)$ if the local orientation of $f(\Sigma)$ followed by the local coorientation of $f(\Sigma)$ agrees with the standard orientation of $N$. 
 \begin{definition}
 An integral (respectively, $\Z_2$-valued) chessboard function $c$ defines a \emph{canonical local coorientation} on $f(\Sigma)$ in the direction of the region over which $c$ assumes the smaller value (respectively, the even value).  The local orientation that agrees with the canonical local coorientation is said to be a \emph{canonical local orientation}. 
\end{definition}

Let $f\co M\to F$ be a stable map of a closed manifold of dimension $m\ge 2$ to an oriented surface $F$, and $c$ a chessboard function, either integral or $\Z_2$-valued. Then, the pair $(f,c)$ gives rise to a diagram $(\Sigma(f); P, Q)$, where $\Sigma(f)$ is the singular set of the map $f$,  and the subsets $P$ and $Q$ of distinguished points are the sets $A_2(f)$ and $A_{11}(f)$ respectively. The pairs $(q_i, q_i')$ of points in $Q$ are the fold points with the same image in $N$, i.e. the self-intersection points.  Finally,  the orientation of the arcs of $\Sigma(f)$ is the one induced by the canonical local orientation of $f(\Sigma)$. 

\begin{proposition}\label{p:10a} Let $f\co M\to F$ be a generic map of non-negative dimension of a closed manifold $M$ to an oriented surface $F$, and $c$ an integral or $\Z_2$-valued chessboard function on $F \setminus f(\Sigma)$. Then $(\Sigma(f); P, Q)$ is an oriented singular set diagram, where $\Sigma(f)$ is equipped with the canonical local orientation. 
\end{proposition}
\begin{proof} Given a generic map $f\co M\to F$ of a manifold $M$, we have defined a manifold $\Sigma(f)$,  together with two families of points $P$ and $Q$ that break $f(\Sigma)$ into canonically oriented arcs. By Lemma~\ref{l:11d} below, the orientations of arcs that share a common point in $P$ agree. By Lemma~\ref{l:14d} below, if $q_j$ is a common point of arcs $\alpha$ and $\beta$, and $q_j'$ is a common point of arcs $\alpha'$ and $\beta'$, then the orientations on the arcs $\alpha$ and $\beta$ agree if and only if the same is true for the arcs $\alpha'$ and $\beta'$. Thus, indeed, each generic map $f$ of non-negative dimension,  together with a chessboard function,  defines an oriented singular set diagram.  To complete the proof of Proposition~\ref{p:10a}, it remains to provide proofs of Lemma~\ref{l:11d} and Lemma~\ref{l:14d}. 

\begin{lemma}\label{l:11d} Let $\alpha$ and $\beta$ be two arcs in $\Sigma(f)$ that share a common endpoint $p\in P$. Then, the canonical orientations of arcs $\alpha$ and $\beta$ agree. 
\end{lemma}

\noindent Notice that in the statement of Lemma~\ref{l:11d},  we do not require that $\alpha$ and $\beta$ are distinct. 

\begin{proof}
Consider a neighborhood  $W\subset F$  of the image of a cusp point $p \in P$.  We may assume that the curve $(f(\alpha) \cup f(\beta)) \cap W$ splits $W$ into two regions.  The coorientation of $f(\alpha)$ and $f(\beta)$ are in the direction of the region where the chessboard function assumes the smaller value for integer chessboard functions, and an even value for $\Z_2$-valued chessboard functions.  In particular,  these coorientations agree. Thus, the orientations of $\alpha$ and $\beta$ agree. 
\end{proof}

Suppose now that $\alpha, \alpha', \beta, \beta'$ are four arcs in $\Sigma(f)$,  such that $\alpha$ and $\beta$ share a common endpoint $q\in Q$, while $\alpha'$ and $\beta'$ share a common endpoint $q'\in Q$,  where $q$ and $q'$ are paired points,  i.e.  $y=f(q)=f(q')$.  Then,  the curve $f(\Sigma)\cap U$  breaks a neighborhood $U$ of $y$ in $F$ into four regions.  We call these regions $L, T, R, B$ for left,  top,  right, and bottom,  respectively.

Let $(a, b, c, d)$ be the values of the chessboard function $c$ at four points that are respectively in the regions $L, T, R, B$,  e.g., see Fig.~\ref{fig:7}.  We note that the order of entries $(a,b,c,d)$ depends on the choice of, say, the left region $L$. However, up to cyclic permutations, the tuple $(a,b,c,d)$ is an invariant of the double point, called the \emph{type} of the double point. 

\begin{lemma}\label{l:13} Each type of double points is either of the form  $(a, a+1, a, a-1)$ or $(a, a+1, a, a+1)$  for some $a$, where $a$ is a non-negative integer if the chessboard function is integral, and it is an element of $\Z_2$ if the chessboard function is $\Z_2$-valued.  
\end{lemma}

The proof of Lemma~\ref{l:13} is straightforward; we omit it.  We note that if the chessboard function is $\Z_2$-valued, then the two types of double points in Lemma~\ref{l:13} are the same.

\iffalse
\begin{proof}
Without loss of generality,  we may assume that $c$ is maximal over the region $R$,  say $c\equiv a+1$ over $R$.  Then, the value of $c$ over the adjacent regions $T$ and $B$ is $a$. There are only two values that the function $c$ may take over the region $L$,  namely $a-1$ or $a+1$.  These cases correspond to the types $(a, a+1, a, a-1)$ and $(a, a+1, a, a+1)$  respectively. 
\end{proof}
\fi

\begin{figure}
\includegraphics[width=0.8\textwidth]{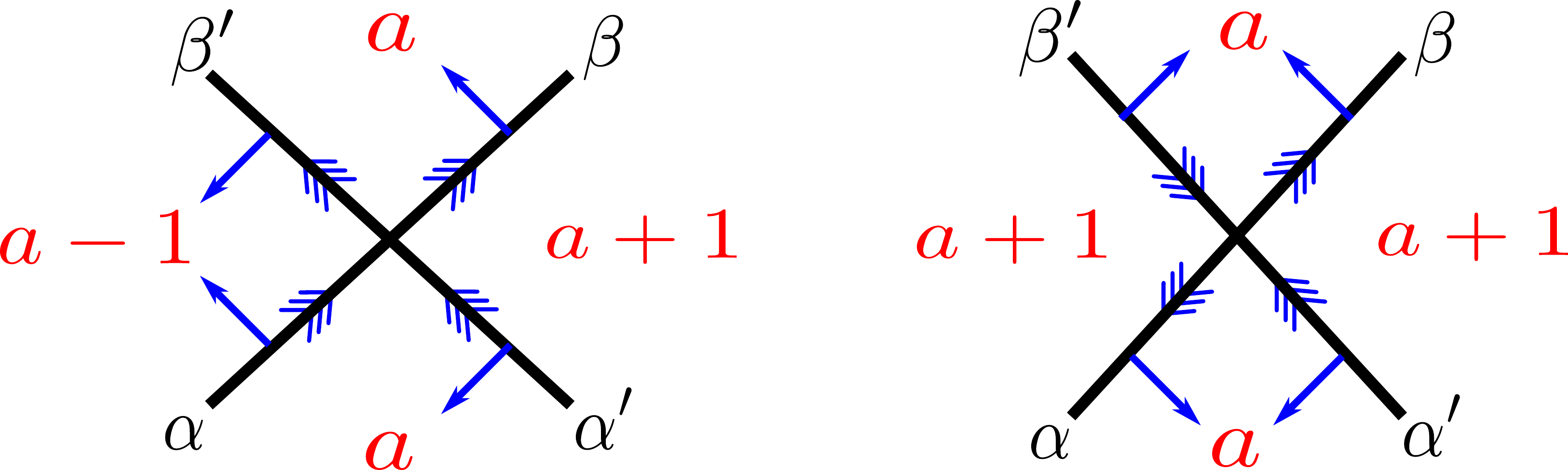}
\caption{Coorientation of arcs near double points of types $(a-1, a, a+1, a)$ on the left, and $(a+1, a, a+1, a)$ on the right.}
\label{fig:7}
\end{figure}

\begin{lemma}\label{l:14d}  The canonical orientations of $\alpha$ and $\beta$ agree if and only if the canonical orientations of $\alpha'$ and $\beta'$ agree.  
\end{lemma}
\begin{proof} We will give an argument for an integral chessboard function; for a $\Z_2$-valued chessboard function the argument is similar. 
Without loss of generality,  we may assume that the arcs are labeled $\alpha, \beta, \alpha', \beta'$  as in Fig.~\ref{fig:7}.  By Lemma~\ref{l:13},  the type of the double point is either of the form $(a, a+1, a, a-1)$ or $(a, a+1, a, a+1)$.  If  the double point is of the form $(a, a+1, a, a-1)$,  then the values of $c$ are as shown on the left schematic of Fig.~\ref{fig:7}. Therefore, the coorientations, and hence orientations, of $\alpha$ and $\beta$ agree.  Similarly,  the orientations of $\alpha'$ and $\beta'$ agree.  If the double point is of the form $(a, a+1, a, a+1)$,  then the values of $c$  are as on the right schematic of Fig.~\ref{fig:7}, and therefore,  the coorientations, and hence orientations, of $\alpha$ and $\beta$ do not agree.  Similarly, the orientations of $\alpha'$ and $\beta'$ disagree,  as well.  
\end{proof}

This completes the proof of Proposition~\ref{p:10a}.
\end{proof}

\section{Examples of Chessboard functions}\label{s:7}

 In this section we give several examples of integral chessboard functions. We note that the reduction modulo 2 turns any integral chessboard function into a $\Z_2$-valued chessboard function.

\subsection{The chessboard function for maps of dimension $0$ counting path components of the fiber.}

Let $f\co M\to N$ be a generic map of a closed manifold of dimension $n$ to an oriented manifold of dimension $n$. We say that the map $f$ is of \emph{odd degree} if the number of points in the inverse image of any regular value of $f$ is odd. Otherwise, we say that $f$ is of \emph{even degree}.  For a regular value $y \in N$ of $f$, let $\#|f^{-1}(y)|$ denote the number of path components in the fiber $f^{-1}(y)$.  Consider the following integer-valued function:
\[ c(y) = \begin{cases} 
              \frac{\#|f^{-1}(y)|}{2}  & \text{if } f \text{ is of even degree}, \\
               \frac{\#|f^{-1}(y)|+1}{2}  & \text{if } f \text{ is of odd degree}. 
              \end{cases}
\]
 It immediately follows that $c$ is an integral chessboard function.

\subsection{The chessboard function for maps of dimension $1$ counting path components of the fiber.}\label{ss:7.2}

Let $f\co M\to N$ be a generic map of a closed oriented manifold of dimension $n+1$ to an oriented manifold $N$ of dimension $n$. For a regular value $y \in N$ of $f$, let $c(y)$ denote the number of path components in the fiber $f^{-1}(y)$, i.e.  
\[
c(y) = \#|f^{-1}(y)|
\]
We claim that $c(y)$ is a chessboard function on $N\setminus f(\Sigma)$. Indeed, let $z$ be a fold singular value of $f$ that is not a self-intersection point of $f(\Sigma)$. Then there is a disc neighborhood $U \ni z$ such that $U\setminus (U\cap f(\Sigma))$ consists of two open discs $U_1$ and $U_2$.

\begin{lemma}\label{l:9} Suppose that $f:M \to N$ is a generic map of an oriented closed manifold of dimension $n+1$ to an oriented manifold $N$ of dimension $n$. Let $y_1\in U_1$ and $y_2\in U_2$ be two points. Then,  the number of path components in the fiber $f^{-1}(y_1)$ differs from the number of path components in the fiber $f^{-1}(y_2)$ precisely by $1$, i.e.  
\[
\#|f^{-1}(y_2)| = \#|f^{-1}(y_1)| \pm 1
\]
\end{lemma}
\begin{proof} Without loss of generality, we may assume that $U \cong D^{n-1} \times (-1,1)$, while $f(\Sigma)\cap U$ coincides with $D^{n-1}\times \{0\}$, where $D^{n-1}$ is a disc of dimension $n-1$. Let $\gamma$ denote the embedded curve $\{0\}\times (-1,1)$. We may assume that $y_1$ and $y_2$ are points on $\gamma$.  Now, let $\pi_1:U \to D^{n-1}$ denote the projection of $U$ onto the first factor. Then the composition $\pi_1\circ f|_{M_0}\co f^{-1}(U)\to D^{n-1}$ is a proper submersion of the manifold $M_0:=f^{-1}(U)$. Indeed, since both $f|M_0$  and $\pi_1$ are proper, we deduce that their composition is also proper. To show that $f$ is a submersion we need to show that the homomorphism $d_x(\pi_1\circ f|_{M_0})$ is surjective for each point $x\in M_0$. We have
\[
       \mathop\mathrm{Im} d(\pi_1\circ f|_{M_0}) = d\pi_1(\mathop\mathrm{Im} d(f|_{M_0})).
\]
If $x\in M_0$, then both $d_xf$ and $d_{f(x)}\pi_1$ are surjective, and therefore their composition $d_x(\pi_1\circ f|_{M_0})$ is also surjective. If $x\in \Sigma(f)$, then the image of $d_xf$ is $T_{f(x)}(D^{n-1}\times \{0\})$. The restriction of $d\pi_1$ to this space is a surjective map. Thus, again  the homomorphism $d_x(\pi_1\circ f|_{M_0})$ is surjective.
Consequently, the map $\pi_1\circ f|_{M_0}$ is a trivial fiber  bundle with fiber diffeomorphic to $V:=f^{-1}(\gamma)$, i.e.  $M_0\cong V\times D^{n-1}$. In view of the inherited orientation on $M_0$, we deduce that the manifold $V$ is also orientable. 

Now,  we examine the number of components of the preimages $f^{-1}(y_1)$,  $f^{-1}(y_2)$ which are subsets of the surface $V$.  Since the restriction $f|_V\co V\to \gamma$ is a Morse function, the manifold $f^{-1}(y_2)$ is obtained from $f^{-1}(y_1)$ by an elementary oriented surgery. We conclude that the numbers of path components in $f^{-1}(y_1)$ and $f^{-1}(y_2)$ differ by exactly $1$. 
\end{proof}

Lemma~\ref{l:9} shows that the function $c$ counting the number of path components in the regular fibers of $f$ is an integral chessboard function. In particular, the image of the singular set $f(\Sigma)$ carries a canonical local coorientation.

\subsection{The Euler chessboard function}

Let $f\co M\to N$ be a generic map of a closed manifold of dimension $n+2q$ for some $q\ge 0$ to an oriented manifold $N$ of dimension $n$. Let $c$ be the following continuous integer valued function on $\R^n\setminus f(\Sigma)$:

\[
     c(y) = \begin{cases} 
              \frac{\chi(f^{-1}(y))}{2}  &  \text{if } \chi(f^{-1}(y)) \text{ is even}, \\
              \frac{\chi(f^{-1}(y))+1}{2} & \text{if } \chi(f^{-1}(y)) \text{ is odd}. \\
              \end{cases}
\]

Recall that under elementary surgery,  the Euler characteristic of fibers in adjacent regions is changed by $\pm 2$.  From this fact, it follows that $c$ is in integral chessboard function. 

\subsection{The depth function}

Let $f\co M\to \R^n$ be a generic smooth map of a closed manifold, where $n>1$. Given a point $y\in \R^n\setminus f(\Sigma)$,  we say that a path $\gamma$ is a \emph{path to infinity} if one endpoint of $\gamma$ is contained in the unbounded region of $\R^n\setminus f(\Sigma)$. We note that since $n>1$, the unbounded region is unique.  Also, we say that a path $\ell_y$ from $y$ to infinity is a \emph{generic curve} if it intersects each stratum $f(\Sigma^I)$ of the singular set transversely,  and the intersection $\ell_y \cap f(A_{11})$ is empty.  We note that a generic curve $\ell_y$ is disjoint from the strata $f(\Sigma^I)$ of dimension $\le n-2$. Consequently, the curve $\ell_y$ only intersects the singular set $f(\Sigma)$ at fold critical values, i.e., the intersection $\ell_y\cap f(\Sigma)$ is a subset of $f(A_1)$. 

The depth function $d\co \R^n\setminus f(\Sigma)\to \Z_{\ge 0}$ associates with each point $y$  the minimal number of intersection points $\ell_y\cap f(\Sigma)$,  where $\ell_y$ ranges over all generic paths from $y$ to infinity.  For estimates of the invariant 
\[
\mathop\mathrm{dep}(\Sigma)=\min\{d(y)\, |\, y\in \R^n\setminus f(\Sigma)\}
\]
 we refer the reader to  \cite{Gr}.  
 
 \iffalse
 {\color{blue} I think that the following is too obvious.}
 
{\color{red} Now, let $R_y \subset \R^n\setminus f(\Sigma) $ denote the region containing the regular value $y$.  We define the depth of the region $R_y$ as 
\[
\mathop\mathrm{dep}(R_y)=\min\{d(y)\, |\, y\in R_y\}
\]
We note that for any other point $y_1$ in the region $R_y$,  we have that $R_{y_1} = R_y$ and thus $\mathop\mathrm{dep}(R_{y_1}) = \mathop\mathrm{dep}(R_y)$.  From this, it is clear that in the above definition, the value of $\mathop\mathrm{dep}(R_y)$ does not depend on choice of $y \in R_y$.  
The following lemma proves that the function defined by 
\[
c(y) = \mathop\mathrm{dep}(R_y)
\]
is a chessboard function on $\R^n \setminus f(\Sigma)$. 
 }
 \fi
 
 \begin{lemma}\label{l:tech} Let $f\co M\to \R^n$ be a smooth generic map of a closed manifold of dimension $m\ge n$. Suppose that $n>1$. Let $\gamma \co [-1,1]\to \R^n$  be a smooth embedded curve with image in $(\R^n\setminus f(M))\cup f(A_0)\cup f(A_1)$. Suppose that $\gamma$ intersects $f(A_1)$ transversely at a unique point $\gamma(0)$,  and define $y=\gamma(1)$ and $z= \gamma(-1)$.  Then $d(y) = d(z) \pm 1$.  
 \end{lemma} 
 
 \begin{proof}  Let $X$ denote the set of singular points $x\in \Sigma^I(f)$ of types $I=(m-n+1, 0)$ and $(m-n+1, 1, 0)$. Then $f(X)\subset \R^n$ is a submanifold of codimension $1$ with cusps and $f(\Sigma \setminus X)$ is a finite union of submanifolds of $\R^n$ of codimension at least $3$. Indeed, the set $\Sigma(f)$ is the union of sets $\Sigma^{i}(f)$, which consist of points $x$ at which the kernel rank is $i$,  where $i= m-n+1 ,..., m$. If $f$ is generic, then each $\Sigma^i(f) \subset M$ is a submanifold of codimension $i(n-m+i)$. In particular, if $i\ge m-n+2$, then the codimension of $\Sigma^{i}(f)$ is at least $4$.  Similarly, by the Boardman formula \cite[\S 2.5]{AVGZ},  the codimension of $\Sigma^{i_1, i_2, ,..., i_k}(f)$ is
  \[
   \nu_{i_1,..., i_k}(m,n)=(n-m+i_1)\mu(i_1,..., i_k)-(i_1-i_2)\mu(i_2,..., i_k)-...-(i_{k-1}-i_k)\mu(i_k),
  \] 
  where $\mu(i_1,..., i_k)$ is the number of sequences $j_1,.., j_k$ of non-negative integers such that $j_1\ge j_2\ge ...\ge j_k$,  and $i_1\ge j_1>0$, $i_2\ge j_2$, ..., $i_k\ge j_k$. Thus, the codimension of a singular stratum $f(\Sigma^I) \subset \R^n$ is at most $2$ if and only if $I$ is $(m-n+1,0)$, or $(m-n+1, 1, 0)$. 
 
Now,  let $\gamma \subset \R^n$ be a closed curve intersecting $f(\Sigma)$ transversely at a unique point. Assume,  contrary to the conclusion of Lemma~\ref{l:tech}, that $y=d(\gamma(-1))$ does not differ from $z=d(\gamma(1))$ by $1$. Let $\ell_y$ and $\ell_z$ be respective paths from $y$ and $z$ to infinity that intersect $f(\Sigma)$ transversely precisely $d(y)$ and $d(z)$ times.  Without loss of generality, we may assume that the path $\ell_y^{-1}*\gamma*\ell_z$ is closed,  where $*$ is path concatenation.  It is important to note that this closed path is null-homotopic.  Furthermore, without loss of generality,  we may assume that $\ell_y^{-1}*\gamma*\ell_z$ avoids $f(\Sigma\setminus X)$ for all moments of time during the homotopy to a point. Thus,  under the specified generic homotopy of $\ell_y^{-1}*\gamma*\ell_z$,  the number of intersection points of $\ell_y^{-1}*\gamma*\ell_z$ with the stratified manifold $f(X)$ changes by an even  number as generically it changes only under finger moves and their inverses.  Therefore,  the number $d(y)+d(z)+1$ of intersection points of $\ell_y^{-1}*\gamma*\ell_z$ with $f(\Sigma)$ is even.  On the other hand,  by definition of the depth function, it is clear that $d(y)$ differs from $d(z)$ by at most $1$. Thus, $d(y)$ differs from $d(z)$ precisely by $1$. 
\end{proof}
  %We will show that for a generic map $f\co M\to \R^2$ of a closed oriented manifold of dimension $3$, the residue class $d(y)\  \mod 2$ equals the residue class of the number of path components in $f^{-1}(y)$. If the manifold $M$ is non-orientable, this is no longer true.  

The depth function can also be defined for any smooth generic map $f\co M\to N$ of a closed manifold of dimension $m$ to a pointed manifold of dimension $n\le m$. In this case, a path to infinity is a path $\gamma$ with an endpoint at the distinguished point of $N$. 

We note that the proof of Lemma~\ref{l:tech} remains valid for maps to simply connected manifolds $N$. 

\section{The cumulative winding number}

\iffalse
{\color{blue} I do not understand well the new paragraph, especially the last sentence.}

{\color{red}To complete preparations for proving the main results,  we introduce and calculate values of the cumulative winding number of an oriented singular set diagram. The cumulative winding number is dependent on fixing a chessboard function which, in turn, fixes the local canonical orientations and co-orientations of the singular set.  We now see how to define the cumulative winding number,  once such a canonical local orientation is established. 
\fi

We first recall the definition of the Gauss map.  A \emph{parallelized manifold} is a manifold $N$ of dimension $n$ together with a smooth map $\tau\co TN\to \R^n$ that restricts to an isomorphism $T_xN\to \R^n$ for each points $x\in N$.
 Given an immersion $\gamma\co [a, b]\to F$ of a segment into a parallelized surface, the \emph{Gauss map} $G\co [a, b]\to S^1$ associates with a point $t\in [a, b]$ the unit vector $\tau(\dot{\gamma}(t))/|\tau(\dot{\gamma(t)})|$. Let $\R\to S^1=[0, 1]/_\sim$, where $\{0\} \sim \{1\}$,  be the universal covering that takes a point $x$ to its congruence class modulo $1$. Let $\tilde G$ denote a lift of $G$ with respect to the universal covering. We define the \emph{winding number} of $\gamma$ by $\tilde G(b)-\tilde G(a)$.  
Given two parametrizations $\gamma'$ and $\gamma$ of the same immersed curve, it follows that the winding numbers of $\gamma'$ and $\gamma$ are the same if and only if the orientations of the curve induced by $\gamma$ and $\gamma'$ agree.

Let $f$ be a generic map to a parallelized surface $F$. Given a chessboard function $c$ on $F\setminus f(\Sigma)$, let $(\Sigma(f); P, Q)$ denote  the diagram associated with $f$. 
Let $\alpha$ be an arc of the diagram $(\Sigma(f); P, Q)$. It corresponds to an arc $\bar{\alpha}=f(\alpha)$ contained in the set $f(\Sigma)$. The curve $\bar{\alpha}$ is an immersed curve in $F$, with possible self-intersection points only on the boundary. By definition, the \emph{winding function} $\varphi$ is a function on the set of arcs of $f(\Sigma)$ that associates with an arc $\alpha$ the winding number $\varphi(\alpha)$ of the curve $\bar{\alpha}$.

\begin{definition}  Suppose that at every self-intersection point of $f(\Sigma)$ the two intersecting segments are perpendicular. Then the real number
\[
 \omega(f) := \  \sum \limits_{\alpha} \varphi(\alpha)
\]
is the \emph{cumulative winding number} of $f(\Sigma)$,  where $\alpha$ ranges over all arcs of $\Sigma(f)$. 
\end{definition}

\begin{proposition}\label{l:10d}
For a generic smooth map $f:M \to F$ of a closed manifold $M$ of dimension $m\ge 2$ to a parallelized surface $F$,  we have
 \[
 \omega(f) \in \frac{1}{2}\Z.
 \]
\end{proposition}

To prove Proposition~\ref{l:10d} we introduce the notion of a regularization of the singular set. The regularization of the singular set $f(\Sigma)$ is a smooth embedded closed curve $\Re f(\Sigma)\subset F$ obtained from $f(\Sigma)$ by smoothing the curve $f(\Sigma)$ near the cusp points  as in Fig.~\ref{fig:8}, and modifying $f(\Sigma)$ near its self-intersection points. Namely, let $y$ be a self-intersection point of $f(\Sigma)$. Then near $y$ the curve $f(\Sigma)$ consists of four arcs $\alpha, \alpha'$ and $\beta, \beta'$.  We remove the four arcs $\alpha$, $\beta$, $\alpha'$ and $\beta'$ from $f(\Sigma)$ and attach two new arcs so that the orientation on $f(\Sigma)\setminus \{\alpha\cup \beta\cup \alpha'\cup \beta'\}$ extends over the new attached arcs, see  Fig.~\ref{fig:9}, \ref{fig:10}, and \ref{fig:11}.

\begin{figure}
\includegraphics[width=0.6\textwidth]{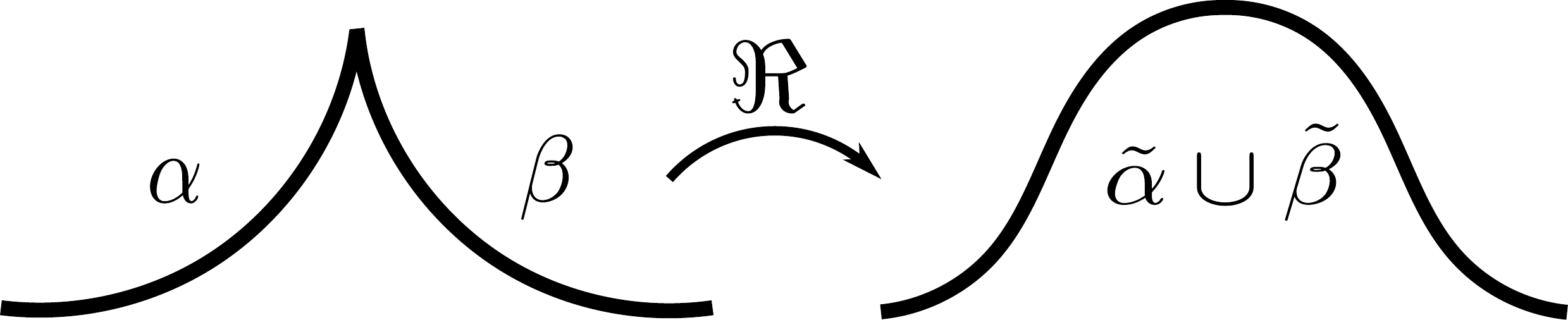}
\caption{Regularization of a Cusp}
\label{fig:8}
\end{figure}

\begin{figure}
\includegraphics[width=0.6\textwidth]{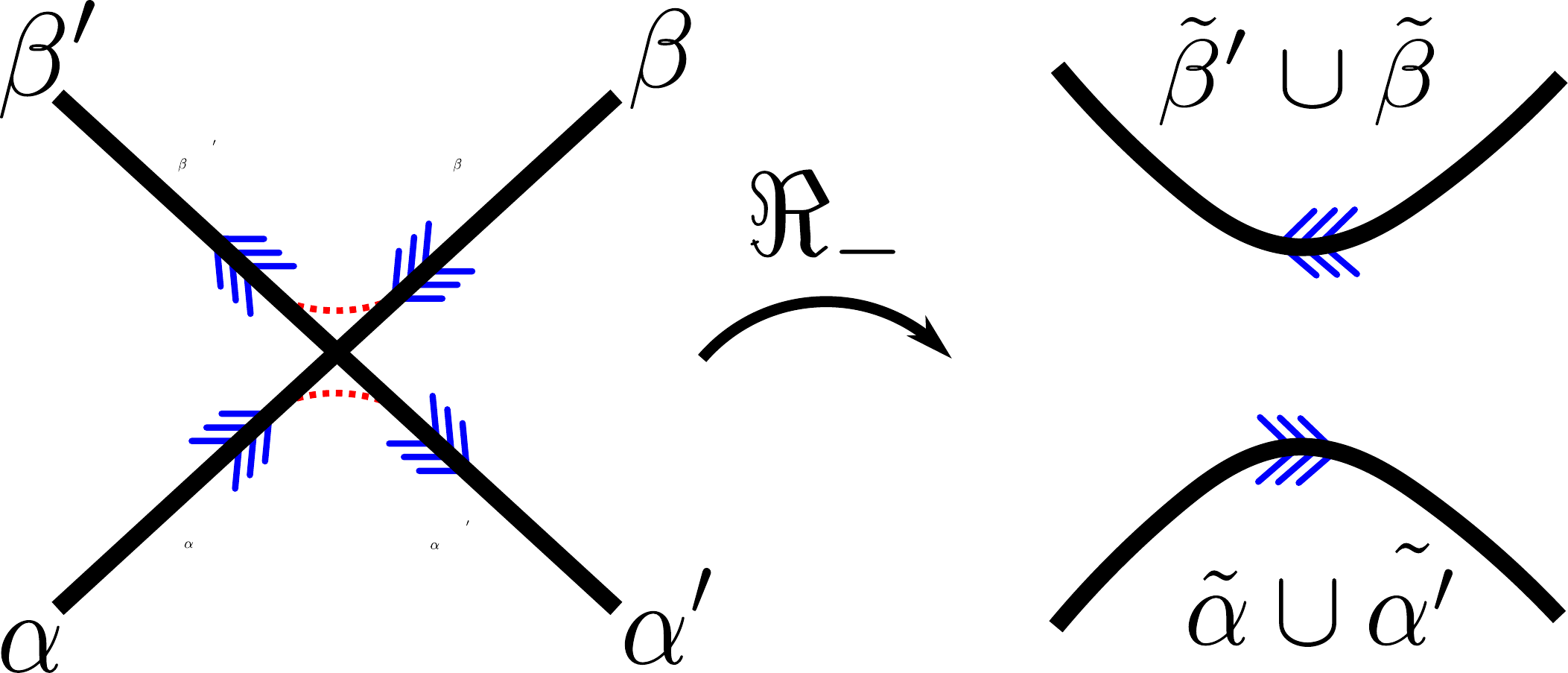}
\caption{Negative Regularization of a Double Point of the form $(a,a+1,a,a+1)$}
\label{fig:9}
\end{figure}

\begin{figure}
\includegraphics[width=0.6\textwidth]{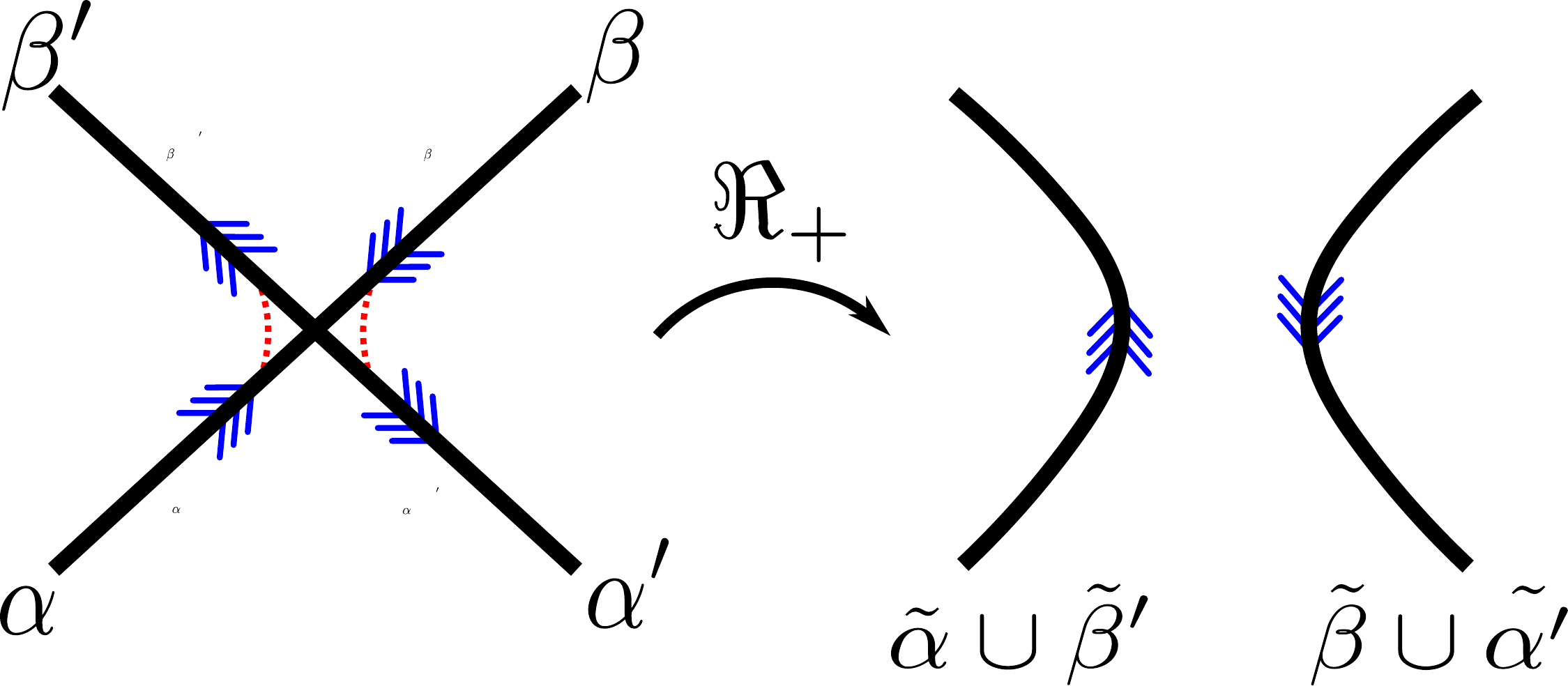}
\caption{Positive Regularization of a Double Point of the form $(a,a+1,a,a+1)$}
\label{fig:10}
\end{figure}

\begin{figure}
\includegraphics[width=0.6\textwidth]{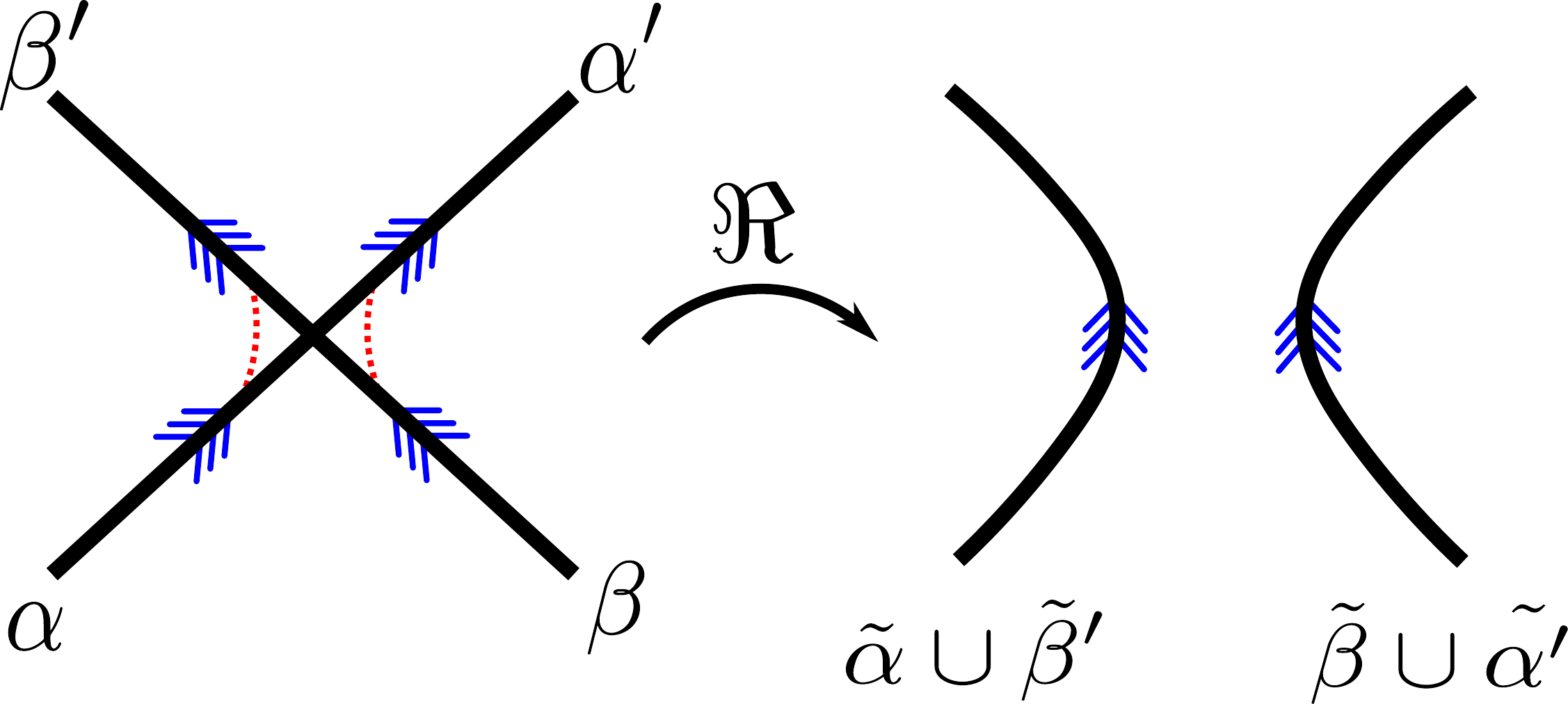}
\caption{Regularization of a Double Point of the form $(a,a+1,a,a-1)$}
\label{fig:11}
\end{figure}

 The proof of the following lemma is omitted as it is straightforward. 

\begin{lemma} The regularization of a cusp decreases the cumulative winding number by $\frac{1}{2}$ if the coorientations of $\alpha$ and $\beta$ are as indicated in Fig.~\ref{fig:15},  and increases the cumulative winding number by $\frac{1}{2}$,  otherwise.
\end{lemma}

\begin{lemma} For a self-intersection point of $f(\Sigma)$ of the form $(a, a+1, a, a+1)$,  there are two regularizations that preserve the orientation of the diagram: $\Re_-$ and $\Re_+$.  The regularizations $\Re_-$ and $\Re_+$ decrease and increase the cumulative winding number by $\frac{1}{2}$  respectively.  For a self-intersection point of the form $(a, a+1, a, a-1)$,  the only possible regularization does not change the cumulative winding number. 
\end{lemma}
\begin{proof} For a self-intersection point of the form $(a, a+1, a, a-1)$,  the only possible regularization does not change the cumulative winding number, see Fig.~\ref{fig:11}.   If a double point is of the form  $(a, a+1, a, a+1)$ there are two possible regularizations that preserve orientation.  One of the regularizations increases the cumulative winding number by $1/2$, while the other one decreases the cumulative winding number by $1/2$, see Fig.~\ref{fig:9}, and Fig.~\ref{fig:10}. The two regularizations are denoted by $\Re_+$ and $\Re_-$ respectively. 
\end{proof}

\begin{proof}[Proof of Proposition~\ref{l:10d}] We note that $\Re f(\Sigma)$ consists of embedded curves, and therefore its cumulative winding number is an integer. On the other hand, under the regularization, the cumulative winding number is changed by $\pm \frac{1}{2}$ for each regularization of a cusp, and $\pm \frac{1}{2}$ or $0$ for each regularization of a self-crossing. 
\end{proof}

\section{Changes of the cumulative winding number under homotopy}

We now observe and record how the cumulative winding number is changed under generic homotopy.    We will denote an $R_2$ move by $R_2(a_1, a_2, a_3, a_4)$, where the quadruple $(a_1, a_2, a_3, a_4)$  encodes the type of the two self-intersection points that are either being created or removed as a result of the $R_2$ move.  For the remainder of our discussion, we adopt the convention that $a_1$ corresponds to the bounded region. In Fig. ~\ref{fig:12},  this is the region bounded by $\alpha_2 \cup \beta_2$.

%We warn the reader that it is possible that $(a_1,a_2,a_3,a_4)$ and $(b_1,b_2,b_3,b_4)$ are double points of the same type, but $R_2(a_1,a_2,a_3,a_4)$ and $R_2(b_1,b_2,b_3,b_4)$ are different. 

 \begin{lemma}\label{l:33}   Let $f\co M\to F$ be a generic map of a smooth closed manifold of dimension $\ge 2$ to a parallelized surface. For any integral chessboard function, there are at most five possible types of $R_2$ moves: 
 $R_2(a, a-1, a-2, a-1)$, $R_2(a,a+1, a+2, a+1)$,  $R_2(a, a+1, a, a-1)$,  
 $R_2(a, a+1, a, a+1)$,  and $R_2(a, a-1, a, a-1)$.  The moves 
 $R_2(a, a-1, a-2, a-1)$, $R_2(a, a+1,a+2, a+1)$ and $R_2(a, a+1, a, a-1)$ do not change $\omega$. The moves $R_2(a, a+1, a, a+1)$ and $R_2(a, a-1, a, a-1)$ change the  cumulative winding number by $1$ and $-1$ respectively. 
 \end{lemma}
 \begin{proof}  Consider an $R_2$-move of type  $R_2(a_1,a_2,a_3,a_4)$. Since the numbers $a_i$ represent the values of an integral chessboard function, we have $a_{i+1}=a_i\pm 1$ and $a_4=a_1\pm 1$. Since up to rotation, the type $R_2(a,a-1,a,a+1)$ is the same as $R_2(a,a+1,a,a-1)$, the list of $R_2$ moves in the statement of Lemma~\ref{l:33} exhausts all possibilities of different types of $R_2$ moves.

 \begin{figure}
\includegraphics[width=0.7\textwidth]{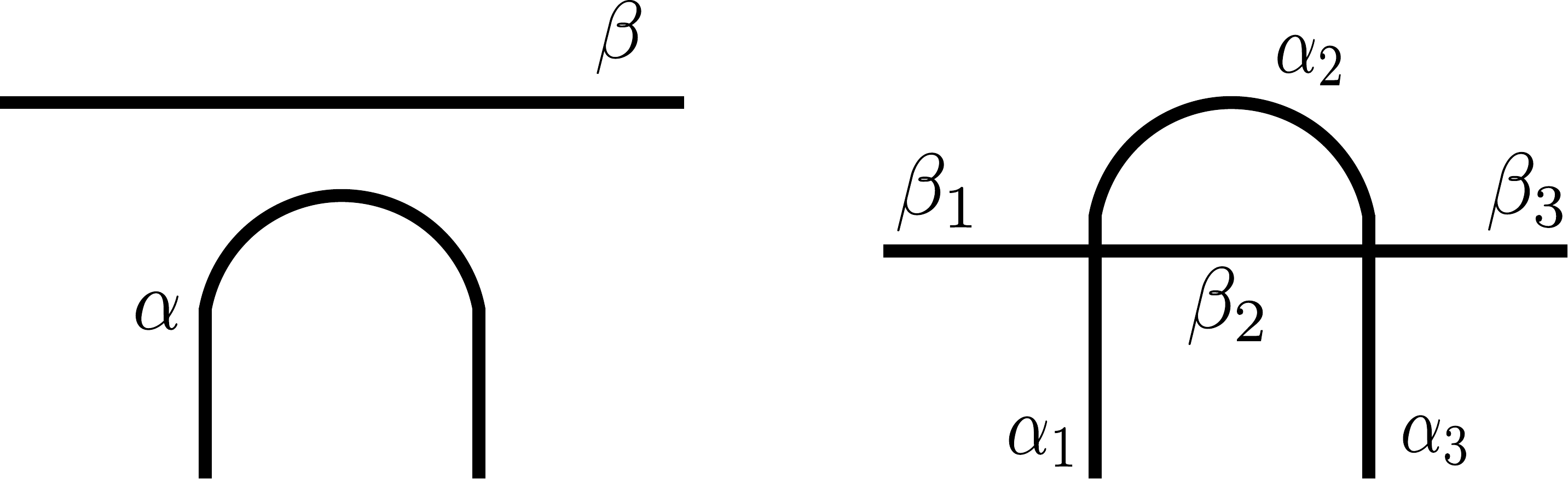}
\caption{Labeled arcs before and after an $R_2$ move}
\label{fig:12}
\end{figure}

 It remains to compute the changes of the cumulative winding number $\omega$ under each $R_2$ type  move.  Denote the two arcs undergoing an $R_2$ move by $\alpha$ and $\beta$.  Without loss of generality,  we assume that $\beta$ is straight and fixed, so that only $\alpha$ moves under homotopy.  After the $R_2$ move,  the two new double points partition the diagram into six arcs: $\alpha_1 ,  \alpha_2 , \alpha_3 , \beta_1 , \beta_2$,  and $\beta_3$ (see Fig~\ref{fig:12}).   We notice that for any type of $R_2$ move,  the winding numbers of $\beta, \alpha_1 , \alpha_3 , \beta_1 , \beta_2$,  and $\beta_3$ are trivial.  Thus, the change in the cumulative winding number $\omega$ is the same as the difference of the winding numbers of $\alpha$ and $\alpha_2$. For example, for the move $R_2(a, a-1, a-2, a-1)$, the winding numbers $\varphi(\alpha)$ and $\varphi(\alpha_2)$ are $-1/2$. Therefore, the cumulative winding number does not change under the $R_2$ move of type $R_2(a, a-1, a-2, a-1)$. The changes in the cumulative winding number for the other $R_2$ moves can be calculated similarly. 
 \end{proof}
 
 Next we turn to the case of swallowtail moves.
 Denote the swallowtail move that creates a self-intersection point of type $(a_1,a_2,a_3,a_4)$ by $ST(a_1,a_2,a_3,a_4)$,  where  $a_1$ corresponds to the bounded region. In Fig. ~\ref{fig:13}, this is the region entrapped by $\alpha_1 \cup \alpha_2 \cup \alpha_3$.  

 For a $\Z_2$-valued chessboard function, we say that an $R_2$ move or an $ST$ move is \emph{even} if the value of the chessboard function over the bounded region is $0$. Otherwise, we say that the $R_2$ move or  $ST$ move is \emph{odd}. 

 \begin{lemma}\label{l:33a} For any $\Z_2$-valued chessboard function, there are at most two $R_2$ moves: even and odd.  An even $R_2$ move increases the cumulative winding number by $1$, while an odd $R_2$ move decreases the cumulative winding number by $1$.  
 \end{lemma}
 
 We omit the proof of Lemma~\ref{l:33a} since the proofs for even and odd $R_2$ moves are the same as those for $R_2(a, a+1, a, a+1)$ and $R_2(a, a-1, a, a-1)$ in  Lemma~\ref{l:33}.

 \begin{lemma}\label{l:32}  Let $f\co M\to F$ be a generic map of a smooth closed manifold of dimension $\ge 2$ to a parallelized surface. For any integral chessboard function
  there are  at most four possible types of swallowtail moves,   
 namely,  $ST(a,a+1,a+2,a+1)$, $ST(a,a+1,a,a+1)$,  $ST(a, a-1, a, a-1)$,  and $ST(a, a-1, a-2,a-1)$.  Moreover,  the moves $ST(a,a+1,a+2,a+1)$ and $ST(a, a-1, a-2,a-1)$ do not change the cumulative winding number. The moves $ST(a,a+1,a,a+1)$ and $ST(a, a-1, a, a-1)$ respectively decrease and increase the cumulative winding number by $\frac{1}{2}$. 
 \end{lemma}
 \begin{proof}  Given a swallowtail type $ST(a_1,a_2, a_3, a_4)$, the numbers $a_i$ represent the values of a chessboard function and therefore satisfy the relations $a_{i+1}=a_i\pm 1$ and $a_4=a_1\pm 1$. Consequently, $ST(a,a+1,a+2,a+1)$, $ST(a,a+1,a,a+1)$,  $ST(a, a-1, a, a-1)$,  and $ST(a, a-1, a-2,a-1)$ are the only possible types of swallowtail moves.
 
  \begin{figure}
\includegraphics[width=0.8\textwidth]{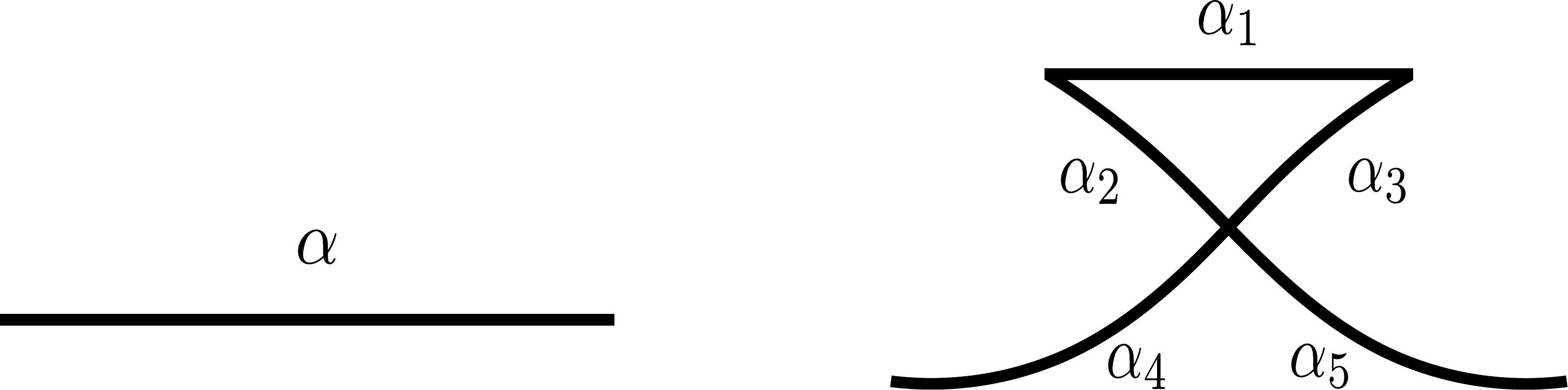}
\caption{Labeled arcs before and after a swallowtail move}
\label{fig:13}
\end{figure}
 
 We now calculate how the winding number is affected by the swallowtail move of type $ST(a,a+1,a+2,a+1)$. Under such a move,  an arc $\alpha$ of the singular set diagram $f(\Sigma)$ is replaced with five sub-arcs:  $\alpha_1, \alpha_2, \alpha_3, \alpha_4$, and $\alpha_5$ (see Fig.~\ref{fig:13}).  Without loss of generality,  we may assume that $\alpha_1$ corresponds to the arc whose endpoints are both cusps.  We may assume that $\alpha$ and $\alpha_1$ are straight,  thus $\varphi(\alpha) = \varphi(\alpha_1)= 0$.  Then
 $\varphi(\alpha_2) = \varphi(\alpha_3) = -\frac{1}{8}$, $\varphi(\alpha_4) = \varphi(\alpha_5) = \frac{1}{8}$, and therefore the cumulative winding number of the singular set does not change under the swallowtail move of type $ST(a, a+1, a+2, a+1)$. 
 
 The change of the cumulative winding number for other types of swallowtail moves can be calculated similarly. 
 \end{proof}

\begin{lemma}\label{l:32a}  For any $\Z_2$-valued chessboard function, there are at most two $ST$ moves: even and odd. An even $ST$ move decreases the cumulative winding number by $1/2$, while an odd $ST$ move increases the cumulative winding number by $1/2$.  
 \end{lemma}

The proofs of Lemma~\ref{l:32a} are the same as those for $ST(a, a+1, a, a+1)$ and $ST(a, a-1, a, a-1)$ in  Lemma~\ref{l:32}.

It remains to examine how the cumulative winding number $\omega(f)$ is changed under wrinkles,  $R_3$ moves, cusp-fold moves,  and cusp merges.  
 
 \begin{lemma} \label{lem:9.3} Let $f\co M\to F$ be a generic map of a smooth closed manifold of dimension $\ge 2$ to a parallelized surface. For any integral chessboard function,
 the wrinkle, cusp merge, and cusp-fold  moves do not change the cumulative winding number associated with the diagram $(\Sigma(f);P,Q)$. %The same is true for cusp merges for the Euler chessboard function,  and the chessboard functions that count path components of regular fibers.
 \end{lemma}
 \begin{proof} 
 The statements of Lemma~\ref{lem:9.3} for wrinkles and cusp merges are easily verified. Next, we examine how cusp-fold moves affect $\omega$. Label the arcs before and after a cusp-fold move as in Fig.~\ref{fig:16}. 
Then  the contribution of $\varphi(\alpha)$ is replaced with $\varphi(\alpha_1)+\varphi(\alpha_2)$,  the contribution of $\varphi(\beta)$ is replaced with $\varphi(\beta_1)+\varphi(\beta_2)$, and the contribution of $\varphi(\gamma)$ is replaced with $\varphi(\gamma_1)+\varphi(\gamma_2)+\varphi(\gamma_3)$.  Consequently, under a cusp-fold move the winding number is modified continuously. Since the cumulative winding number is an element of $\frac{1}{2} \Z$,  we conclude that $\omega$ is unchanged under cusp-fold moves.  
\end{proof}

  \begin{figure}
\includegraphics[width=0.6\textwidth]{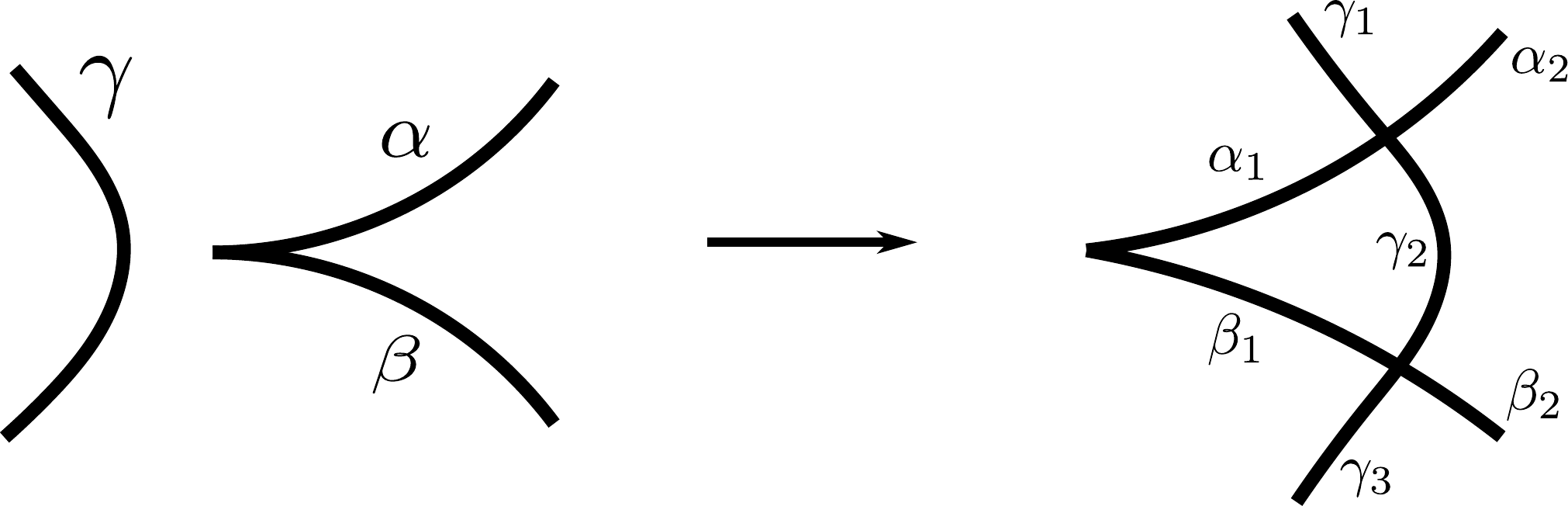}
\caption{Labeling of arcs involved in a cusp-fold move}
\label{fig:16}
\end{figure}

Similarly, we can determine changes of cumulative winding number for $\Z_2$-valued chessboard functions. 

\begin{lemma}  For $\Z_2$-valued chessboard functions,  the  cusp merge,  cusp-fold,  and wrinkle moves do not change the cumulative winding number. 
\end{lemma}

\begin{lemma} For any integral chessboard function, $R_3$ moves change the cumulative winding number by $\pm 1/2$ or $\pm 1$ or $0$. For any $\Z_2$-valued chessboard function, $R_3$-moves change the cumulative winding number by $\pm 1$. 
\end{lemma}
 
The following proposition summarizes the above calculations for $\Z_2$-valued chessboard functions.
 
 \iffalse
\begin{proposition} \label{l:21}   Let $f\co M\to \R^2$ be a generic map of a smooth manifold of dimension $\ge 2$. For any integral  chessboard function,
under generic homotopy of a stable map $f$,  the cumulative winding number $\omega(f)$  may change only under an $ST$, $R_2$ or $R_3$ move.  Under an $ST$ move, the cumulative winding number changes by $\pm \frac{1}{2}$.  Under an $R_2$ move,  the cumulative winding number changes by $\pm 1$.  Under an $R_3$ move, the cumulative winding number changes by $\pm 1/2$,  $\pm 1$ or $0$. 
 \end{proposition}
  \fi
  
\begin{proposition} \label{l:21a} Let $f\co M\to F$ be a generic map of a smooth closed manifold of dimension $\ge 2$ to a parallelized surface. For any $\Z_2$-valued chessboard function,
under generic homotopy of a stable map $f$,  the cumulative winding number $\omega(f)$  may change only under an $ST$, $R_2$ or $R_3$ move.  Under an $ST$ move, the cumulative winding number changes by $\pm \frac{1}{2}$.  Under an $R_2$ or $R_3$ move,  the cumulative winding number changes by $\pm 1$.  
 \end{proposition}

In the rest of the section we prove  Lemmas~\ref{l:22} and \ref{le:9.17}.  To prove Lemmas~\ref{l:22} and \ref{le:9.17}, we will need Lemmas~ \ref{l:23} and \ref{l:25}.

\begin{lemma} \label{l:23} Lef $f\co M\to F$ be a smooth map of a closed oriented manifold of dimension $3$ to an oriented surface. Then for the integral chessboard function of \S\ref{ss:7.2}, 
 the coorientation of arcs in $(\Sigma(f);P,Q)$ that have a cusp endpoint is as on Fig.~\ref{fig:15}. The opposite coorientation is not possible. 
\end{lemma} 
\begin{proof}
Recall that locally a generic map $f\co M \rightarrow \R^2$ is a Morse $2$-function. In particular, for a cusp point $p \in A_2(f)$, we may identify a neighborhood $V$ of $f(p)$ with $[0,1]\times [0,1]$, and  the inverse image $f^{-1}(V)$ with $[0,1]\times M_0$ in such a way that $f|_{f^{-1}(V)}$ is given by $(t, x)\mapsto (t, g_t(x))$, where $g_t$ is a family of generalized Morse functions such that $g_t$ has no critical points for $t\in [0, 1/2)$, $g_{1/2}$ has a unique critical point, and $g_t$ has two canceling Morse critical points for $t\in (1/2, 1]$, see Fig.~\ref{fig:14}.

\begin{figure}
\includegraphics[width=0.4\textwidth]{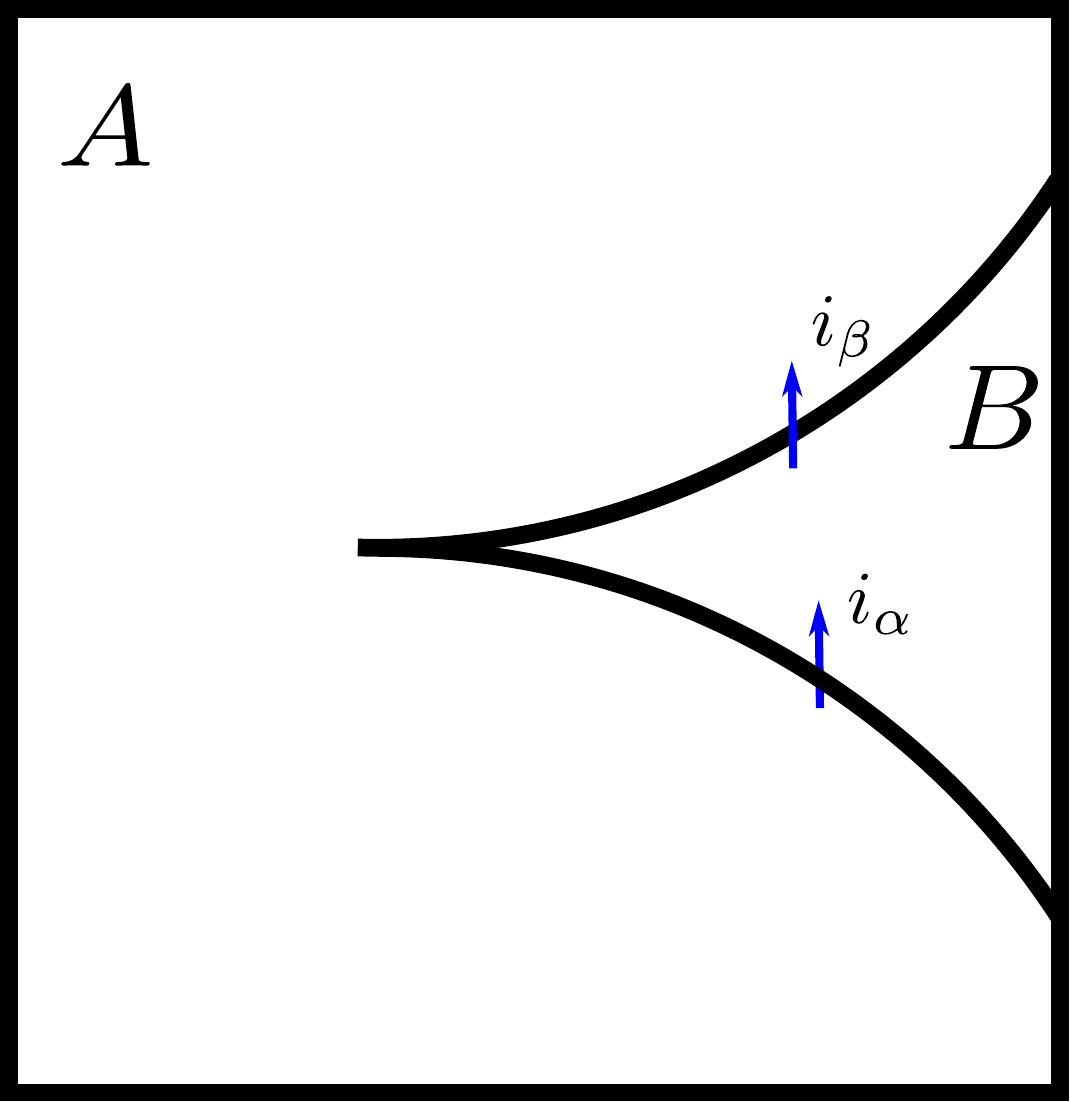}
\caption{The neighborhood V of a cusp point.}
\label{fig:14}
\end{figure}

\begin{figure}
	\includegraphics[width=0.3\textwidth]{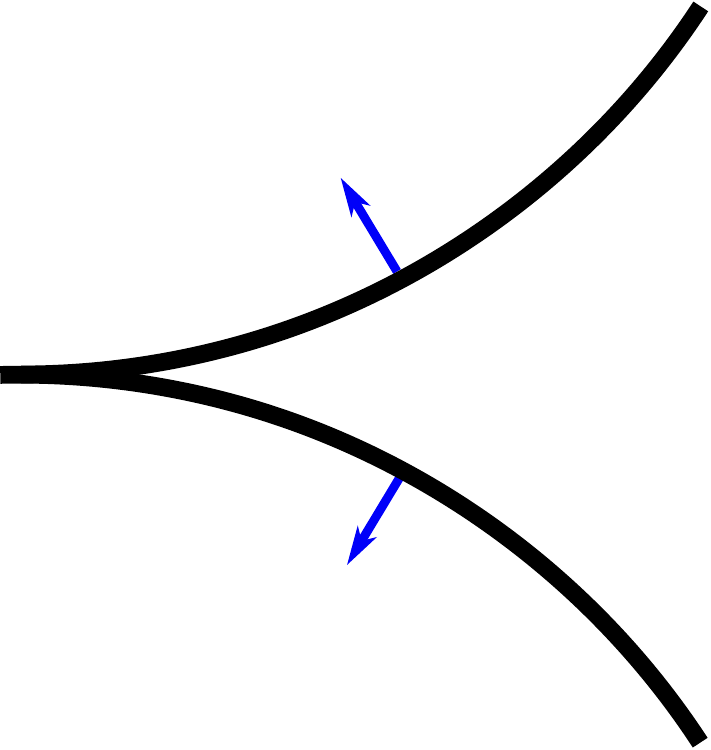}
	\caption{Coorientations of singular arcs near a cusp when M is 3-dimensional}
	\label{fig:15}
\end{figure}

Let $\alpha$ and $\beta$ be two arcs in $f(\Sigma) \cap V$ that share the common cusp endpoint $p \in A_2(f)$.   Then the indices $i_\alpha$ and $i_\beta$ of the two critical points of $g_{3/4}$ on the arcs $\alpha$ and $\beta$ satisfy the relation $i_\beta=i_\alpha+1$.  The arcs $\alpha$ and $\beta$ split $V$ into two regions $A$ and $B$ containing the points $(0, 1/2)$ and $(1,1/2)$ respectively. Both in the case $(i_\alpha, i_\beta)=(0,1)$ and $(i_\alpha, i_\beta)=(1,2)$ the number of path components in the inverse image of any point in $A$ is one less than that of any point in $B$. Therefore, the coorientations of the arcs $\alpha$ and $\beta$ are as on Fig.~\ref{fig:15}. 
\end{proof}

 \begin{lemma}\label{l:25}
 Consider a smooth generic map $f:M \rightarrow F$ of a closed manifold $M$ of even dimension $m\ge 2$ to an oriented surface. When $c$ is the integral Euler chessboard function,  $\Sigma(f)$ does not have self-intersection points of type $(a,a-1,a,a-1)$. 
 \end{lemma}
 
We note that the statement of Lemma~\ref{l:25} is not true for the depth chessboard function.  
 
 \begin{proof}
 The intersecting strands of $f(\Sigma)$ break a neighborhood of a self-intersection point into four regions, which we denote by $R, T, L$ and $B$, for the right, top, left, and bottom regions, respectively.  Note that the diffeomorphism types of the fibers $M_R, M_T, M_L$ and $M_B$ over points in the four respective regions do not depend on the choice of regular values.  If the manifold $M_T$ is obtained from $M_R$ by a surgery of index $i$, then $M_L$ is obtained from $M_B$ by a surgery of the same index $i$. Since $M$ is of even dimension, we conclude
\[ \chi(M_T)-\chi(M_R)=\chi(M_L)-\chi(M_B)=\pm 2. \]
 This rules out the existence of double points of type $(a, a-1, a, a-1)$. 
 \end{proof}

 Recall that a cusp-fold move creates or eliminates two double points of the same type. We will henceforth denote cusp-fold moves creating or eliminating double points of type $(a_1,a_2,a_3,a_4)$ by $CF(a_1,a_2,a_3,a_4)$,  and practice the convention that $a_1$ corresponds to the value of a prescribed chessboard function in the bounded region (in Figure ~\ref{fig:16} this is the region with boundary $\alpha_1 \cup \beta_1 \cup \gamma_2$).  In particular, there are at most two types of cusp-fold moves involving self-intersection points of type $(a, a-1, a, a-1)$, namely, $CF(a, a-1, a, a-1)$ and $CF(a, a+1, a, a+1)$.

 \begin{lemma}\label{l:22}
Let $f,g:M\rightarrow F$ be two homotopic image simple maps of a closed manifold $M$ to an oriented surface $F$.  
\begin{itemize}
\item If $M$ is an oriented manifold of dimension $3$, then for the integral or $\Z_2$-valued chessboard function counting path components of fibers, the number of cusp-fold moves involving self-intersection points of type $(a,a-1,a,a-1)$ is even. 
\item If $\pi_1(F)=1$ and $M$ is of odd dimension $\ge 3$,  then the number of cusp-fold moves involving self-intersection points of type $(a,a-1,a,a-1)$, with respect to the integral or $\Z_2$-valued depth chessboard function, is even.  
\item If the dimension  $m\ge 2$ of $M$ is even,  then for the integral Euler chessboard function there are no cusp-fold moves involving self-intersection points of type $(a, a-1, a, a-1)$. 
\end{itemize}
\end{lemma}

\begin{proof} 
	Suppose $M$ is a closed oriented manifold of dimension $3$ equipped with the integral or $\Z_2$-vaued chessboard function counting path components of fibers. If a cusp fold move $CF(a_1, a_2, a_3, a_4)$ involves a self-intersection point $(a, a-1, a, a-1)$, then $(a_1, a_2, a_3, a_4)$ is obtained from $(a, a-1, a, a-1)$ by a cyclic permutation. In particular, only moves $CF(a, a-1, a, a-1)$ and $CF(a-1, a, a-1, a)$ may involve self-intersection points of type $(a, a-1, a, a-1)$. Furthermore, we claim that  the only possible cusp-fold moves involving self-intersection points of type $(a,a-1,a,a-1)$ are $CF(a,a-1,a,a-1)$. 

Indeed, if the chessboard function is $\Z_2$-valued, then there are no cusp-fold moves except for $CF(a, a-1, a, a-1)$. Suppose now that the chessboard function is integral. Equip 
 $(\Sigma(f);P,Q)$ with the chessboard function counting the number of path  components of regular fibers.  By Lemma ~\ref{l:23},  all cusps are cooriented as in Fig. ~\ref{fig:15}, and therefore, the value of the chessboard function over the bounded region is maximal. Thus, indeed, the only possible cusp-fold move involving self-intersection points of type $(a,a-1,a,a-1)$ is $CF(a,a-1,a,a-1)$.

 Similarly, the only possible cusp-fold moves involving self-intersection points of type $(a,a-1,a,a-1)$ are $CF(a,a-1,a,a-1)$ in the case of maps $f\co M\to F$ of a manifold of arbitrary odd dimension $m\ge 3$ equipped with the integral depth chessboard function as the value of the depth chessboard function over the bounded region of Fig.~\ref{fig:16} is greater than the value over at least over one adjacent region.

On the other hand,  every cusp-fold move  changes the parity of self-intersection points of the fold curve where one intersecting segment of the fold curve has an odd index while  the other one has an even index. No other moves change the parity of the number of such self-intersection points of type $(a, a-1, a, a-1)$. Since $f(\Sigma)$ and $g(\Sigma)$ are embedded, we conclude that the number of $CF(a,a-1,a,a-1)$ moves must be even. 

{\color{red}
\iffalse
\sout{The same argument holds for maps $f:M \rightarrow \R^2$ of a manifold $M$ of an arbitrary odd dimension $m \geq 3$ equipped with the integral depth chessboard function.  Indeed, in the integral case the value of the integral depth chessboard function over the bounded region on Fig.~\ref{fig:15} is greater than or equal to its values over the other regions. Therefore, in this case $CF(a, a-1, a, a-1)$ is the only cusp-fold type move that involves self-intersection points of type $(a, a-1, a, a-1)$.}
\fi

}

Now, let $f:M \rightarrow F$ be a generic map of a manifold $M$ of an arbitrary even dimension $m\geq2$ to an oriented surface.  By Lemma ~\ref{l:25},  there are no self-intersection points of type $(a, a-1,a,a-1)$ with respect to the Euler chessboard function, and therefore,  there are no cusp-fold moves involving self-intersection points of this type at all.  
\end{proof}

\begin{remark} We note that for an arbitrary chessboard function,  its value need not be maximal over the bounded region created by a cusp-fold move.  In general,  there may possibly be six different types of cusp-fold moves:
 $CF(a,a-1,a,a-1), CF(a,a-1,a,a+1), CF(a,a-1,a-2,a-1),CF(a,a+1,a,a+1),CF(a,a+1,a,a-1), $ and $CF(a,a+1,a+2,a+1)$.  
 \end{remark}

 \begin{lemma}\label{le:9.17} Suppose that $f\co M\to F$ is a generic map of a closed manifold of even dimension to a parallelized surface. Then for the integral Euler chessboard function,  the cumulative winding number does not change under homotopy of $f$. 
  \end{lemma}
 \begin{proof}   Consider a map $f$ to $F$. By Lemma ~\ref{l:25}, no double points of type $(a,a-1,a,a-1)$ may occur for the Euler chessboard function.  Consequently, the local coorientation of fold arcs defines a global orientation of the curve of fold points as the coorientations of arcs with common endpoints agree, see Fig.~\ref{fig:7}. Therefore, $R_3$ moves do not change the cumulative winding number. The cumulative winding number is preserved by $R_2$ and $ST$ moves by Lemma~\ref{l:33} and Lemma~\ref{l:32}.
 \iffalse
 In the case where the target space is a punctured surface, we note that the tangent bundle $TF$ is trivial, and therefore the winding number associated with the integral Euler chessboard function is well-defined. The rest of the proof in this case is similar to one in the case where the target space is $\R^2$. 
 \fi
\end{proof}

 \section{Proof of Theorem ~\ref{th:1} and Theorem ~\ref{th:2}}

\begin{reptheorem}{th:1}
\textit{ Let $f$ and $g$ be two homotopic image simple fold maps 
from a closed manifold $M$ of even dimension  $m\geq2$ to an oriented surface $F$ of finite genus.   
Then,  the number of components of $\Sigma(f)$ is congruent modulo two to the number of components of $\Sigma(g)$. }
\end{reptheorem}

\begin{proof}
To begin with let us assume that the target surface is $\R^2$.
 Recall that $\#|\Sigma(f)|$ denotes the number of components of $\Sigma(f)$.  Let $c$ be the integral Euler chessboard function as described in \S7.3. 
 
  By Lemma~\ref{le:9.17}, 
 \[
 \omega(f) \equiv \omega(g) ~(\textrm{mod\ } 2).
 \]
Next, utilizing the hypothesis that $f(\Sigma)$ and $g(\Sigma)$ are embedded, we deduce 
\[
\omega(f) \equiv \#|\Sigma(f)| ~(\textrm{mod\ } 2). 
\]
Combining the previous congruences yields the desired result
\[
 \#|\Sigma(f)| \equiv \omega(f) \equiv \omega(g) \equiv  \#|\Sigma(g)| ~(\textrm{mod\ } 2) .
\]
 This concludes the proof of Theorem~\ref{th:1} in the case of maps to $\R^2$. 
 
 Now, suppose that $F$ is a closed surface. Let $p$ be a point in $F$, away from $f(\Sigma)$. Then the tangent bundle of $F\setminus \{p\}$ is trivial. We fix a trivialization $\tau\co T(F\setminus \{p\})\to \R^2$ of the tangent bundle. Then the winding number $\omega(f)$ is well-defined.   
%Suppose now that the target surface $F$ is a closed surface. Let $p$ be a point in $F$, away from $f(\Sigma(f))$. Then the tangent bundle of $F\setminus \{p\}$ is trivial, and therefore the winding number $w(f)$ is well-defined.  
 %As in the case where the target surface is $\R^2$,  the winding number does not change under generic moves except for swallowtail and Reidemeister-II moves which respectively change the winding number by $\pm \frac{1}{2}$ and $\pm 1$.  
 Under a generic homotopy of $f$,  the curve $f(\Sigma)$ may slide through the point $p$.

 \begin{lemma} \label{l:sliding10.1}As the curve $f(\Sigma)$ slides through the point $p$,  the winding number changes by $\pm \chi(F)$, where $\chi(F)$ denotes the Euler characteristic of the surface $F$. 
\end{lemma}
\begin{proof} Indeed, let $D$ be a small disc in $F$ centered at $p$. Recall that for every nowhere zero vector field $u$ over the boundary of $D$, there is a well-defined winding number which counts the number of rotations of $u(x)$ with respect to a trivialization of the tangent bundle over $D$ as $x$ traverses the boundary of $D$.  It is well-known that the winding number of $u|\partial D$ equals the sum of indices of critical points of any extension of the vector field $u$ over the disc.  Now, let $v$ denote a nowhere vanishing vector field $\tau^{-1}(e_1)$ over $F\setminus \mathop\mathrm{Int}(D)$ trivializing the tangent bundle of $F\setminus \mathop\mathrm{Int}(D)$. It can be extended to a vector field over $F$, which we still denote by $v$. The sum of indices of critical points of $v$ is the Euler characteristic of $F$. Therefore the winding number of $v|\partial D$ with respect to the trivialization of the tangent bundle of $D$ is $\chi(F)$. Consequently, if $w$ is a unit vector field in $TF|\partial D$  that extends to a unit vector field over $D$, then the winding number of $w$ with respect to the trivialization $\tau$ of the tangent bundle of $F\setminus \mathop\mathrm{Int}(D)$ is $\pm \chi(F)$.

Suppose $\{f_t\}$ is a generic homotopy of $f=f_0$, parameterized by $t\in [0,1]$ under which $f(\Sigma)$ slides through the point $p$. Without loss of generality we may assume that $f_0(\Sigma)$ shares a common point with $\partial D$ and has no other common points with $D$. Then 
the curve $f_1(\Sigma)$ is regularly homotopic in $F\setminus \{p\}$ to a smoothening of the concatenation of the curves $f_0(\Sigma)$ and $\partial D$. Thus, up to sign, the difference between the winding numbers of $f_0(\Sigma)$ and $f_1(\Sigma)$ is the Euler characteristic of $F$.   
\end{proof}

 Since the surface $F$ is closed and oriented of genus $g$,  we have $\chi(F)=2-2g$. Thus,  the parity of the winding number is well-defined. Consequently, as in the case when the target surface is $\R^2$, we have
\[
\omega(f) \equiv \omega(g) \  (\mod 2).
\]

This also shows that for every embedded closed curve $\gamma$ on an oriented closed surface $F$, there is a well-defined winding number $\rho(\gamma)\in \Z_2$. The winding number $\rho(\gamma)$ does not depend on the orientation of $\gamma$.  

\begin{lemma}\label{lem:10.1}
Let $\gamma_1$ and $\gamma_2$ be two embedded closed curves on an oriented closed surface $F$. Suppose that $\gamma_1$ and $\gamma_2$ represent the same homology class in $H_1(F; \Z_2)$. Then  
\[
\rho(\gamma_1) - \#|\gamma_1| \equiv \rho(\gamma_2) - \#|\gamma_2| \ (\mod 2),
\]
where $\#|\gamma_i|$ denotes the number of components of $\gamma_i$, for $i=1,2$. 
\end{lemma}
\begin{proof} We may assume that the surface $F$ is connected. 

Recall that an oriented surgery of an embedded closed curve $\gamma$ is embedded if the base of surgery is an embedded strip whose interior is disjoint from the curve $\gamma$, see Fig.~\ref{fig:19} and ~\ref{fig:20}. We note that under each oriented embedded surgery the value $\rho(\gamma)$,  as well as the modulo two residue class of $\#|\gamma|$,  is changed.  Thus, the value  $\rho(\gamma) - \#|\gamma|$ remains the same.

By performing an appropriate number of elementary surgeries, we may assume $\gamma_1$ and $\gamma_2$ are path connected closed embedded curves. Since $\gamma_1$ and $\gamma_2$ represent the same homology class in $H_1(F; \Z_2)$, they are either both separating or non-separating.  
  
  If the curves are non-separating,  then there is a diffeomorphism $\varphi$ of the target surface $F$ to itself that takes $\gamma_1$ to $\gamma_2$.  Thus, the parity of $\rho(\gamma_1) - \#|\gamma_1|$ is the same as the parity of $\rho(\gamma_2) - \#|\gamma_2|$. 
  
  Next, suppose that the curves $\gamma_1$ and $\gamma_2$ are separating.  Without loss of generality, we may assume that $\gamma_1$ and $\gamma_2$ are disjoint, since there is a diffeomorphism $\varphi$ of $F$ such that $\gamma_1$ and $\varphi(\gamma_2)$ are disjoint.     We may always construct a Morse function $h$ on $F$ such that $\gamma_1$ and $\gamma_2$ are two regular level sets of $F$, say $h^{-1}(0)=\gamma_1$ and $h^{-1}(1)=\varphi(\gamma_2)$. Each critical point of $h$ in $h^{-1}[0,1]$ corresponds to an elementary oriented embedded surgery. The composition of these elementary oriented embedded surgeries takes $\gamma_1$ to a curve isotopic to $\gamma_2$.  As mentioned above, the value of  $\rho(\gamma_1)$ and the modulo two residue class $\#|\gamma_1|$ are changed under each elementary oriented embedded surgery.  Therefore,  the value $\rho(\gamma_1) - \#|\gamma_1|$ is preserved.  

In both cases, we have deduced the desired result.
\end{proof}

 \begin{figure}
\includegraphics[width=0.6\textwidth]{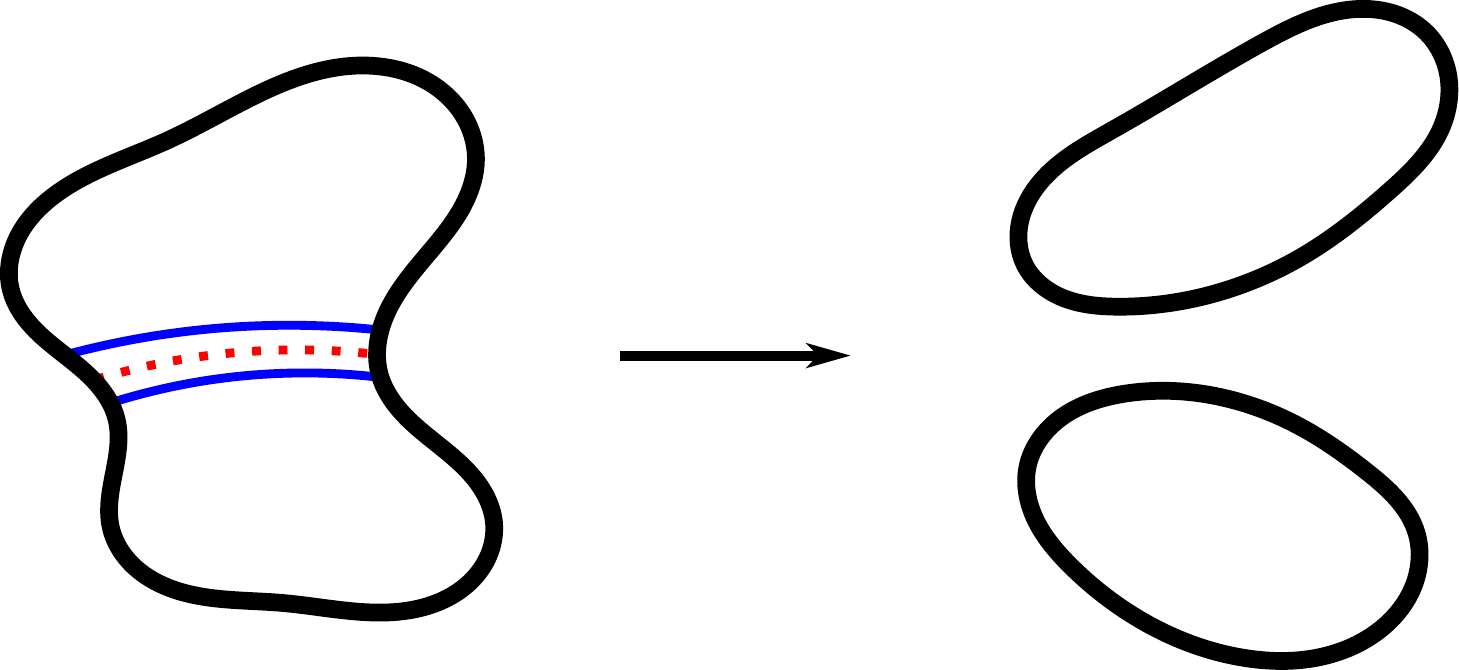}
\caption{Elementary surgery increasing the number of connected components.}
\label{fig:19}
\end{figure}

 \begin{figure}
\includegraphics[width=0.6\textwidth]{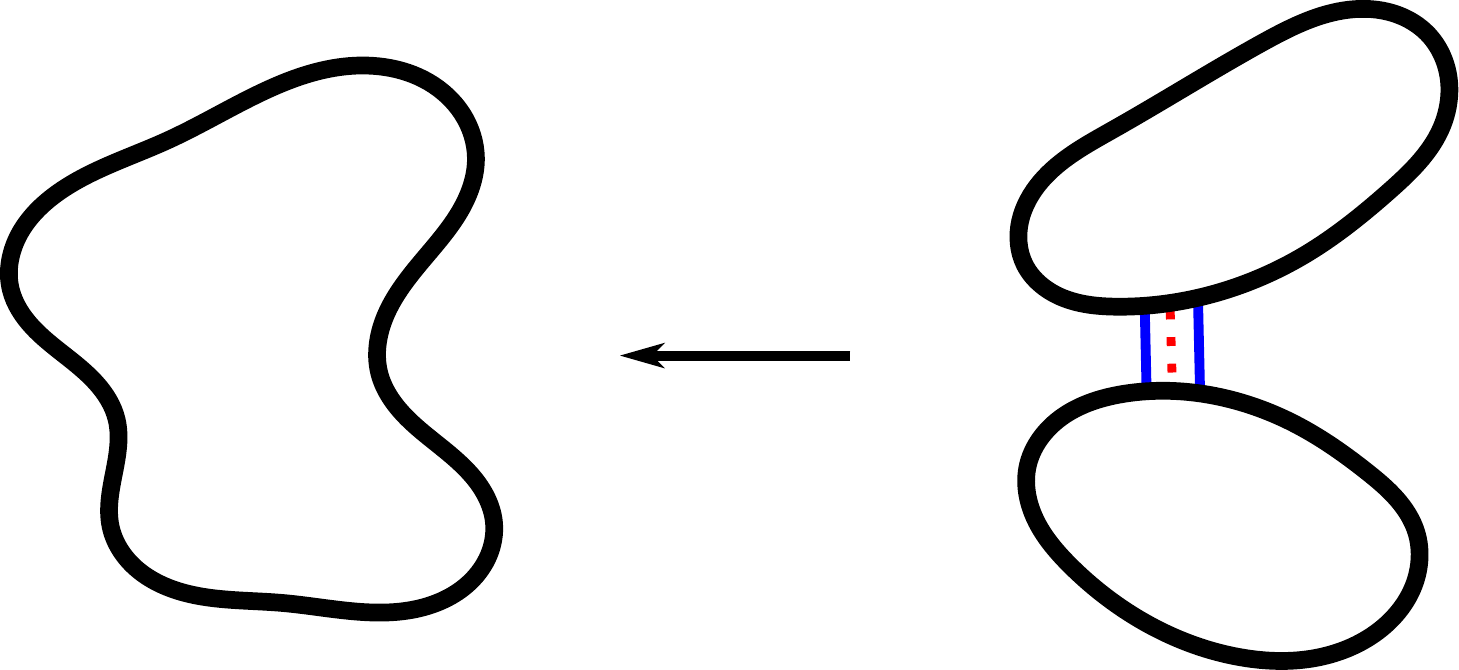}
\caption{Elementary surgery decreasing the number of connected components.}
\label{fig:20}
\end{figure}

In view of Lemma~\ref{lem:10.1},   we conclude that
\[
\#|\Sigma(f)| \equiv \#|\Sigma(g)| \ (\mod 2).
\]

If $F$ is not a closed surface, then it admits an embedding $j$ into a closed surface $F'$. Then the numbers of path components of $\Sigma(f)$ and $\Sigma(g)$ are the same as the numbers of components of $\Sigma(j\circ f)$ and $\Sigma(j\circ g)$,  respectively. Therefore,  the case where $F$ is an open surface of finite genus follows from the case where $F$ is a closed surface. 
\end{proof}

\begin{reptheorem}{th:2}
\textit{ Let $f$ and $g$ be two homotopic image simple fold maps  $M\to F$, where %with no cusp singular points
\begin{itemize}
\item $M$ is a closed manifold of odd dimension $m> 2$ and $F$ is $\R^2$ or $S^2$, or
\item $M$ is a closed oriented manifold of dimension $3$, and $F$ is an oriented surface.  
 \end{itemize}
 %Suppose that $f|_{\Sigma(f)}$ and $g|_{\Sigma(g)}$ are embeddings. Also, 
Suppose that no $R_3$ moves occur during a generic homotopy from $f$ to $g$. Then, the number of components of $\Sigma(f)$ is congruent modulo two to the number of components of $\Sigma(g)$. }
 \end{reptheorem}
 
  \begin{proof} We will work with the $\Z_2$-valued depth chessboard function if $m>2$ and $F$ is $\R^2$ or $S^2$. If $m=3$ and $M$ is oriented, then we will work with the $\Z_2$-valued chessboard function that counts the number modulo $2$ of path components in the preimage of a regular value. 

Since $f$ and $g$ are odd dimensional image simple fold maps to a surface, the homology class of swallowtail singularities of a homotopy $H$ of $f$ to $g$ is trivial, as the image of the set of swallowtail singular points of $H$ in $F\times [0,1]$ bounds the set of double points of $H(\Sigma)$. 
Therefore, by the argument in \cite{Sad}, there is a formal homotopy of $f$ to $g$ with no swallowtail singularities.  
   By the relative h-principle for swallowtail singular points  \cite{An}, we may assume that the (genuine) homotopy of $f$ to $g$ does not have swallowtail singular points. 
   
   By Lemma~\ref{l:22}, the number of cusp-fold moves is even since for $\Z_2$-valued chessboard functions all cusp-fold moves are of type $CF(a, a-1, a, a-1)$. %Indeed, since the dimension of the maps $f$ and $g$ is odd, the parity of the index of any  fold point does not depend on the local coorientation of the fold curve in the target space. Furthermore, the parities of the indices of two fold strands in a neighborhood of any cusp point are different. Consequently, each $CF$ move changes the number of intersection points of the fold curve of even index with the fold curve of odd index by $1$. No other moves change the parity of the number of such intersection points. Therefore, the number of $CF$ moves is even. 
   
    On the other hand, since there are no swallowtail singular points,  the number of pairs of self-intersection points changes under homotopy by 
  \[
   \#|CF|+\#|R_2|\equiv 0 \ (\mod 2). 
  \]
  Consequently, the number of $R_2$ moves is also even.  
  
  Suppose now that $F$ is a parallelized surface. By Proposition~\ref{l:21a}, only swallowtail, $R_2$ and $R_3$ moves may change the cumulative winding number. We have assumed that the homotopy of $f$ to $g$ does not involve swallowtail and $R_3$ moves. Therefore, since each $R_2$ move changes the cumulative winding number by $\pm 1$, and the number of $R_2$ moves is even, we conclude that the parity of the cumulative winding numbers for $f$ and $g$ are the same. Consequently, the number of components of $\Sigma(f)$ is congruent modulo two to the number of components of $\Sigma(g)$. 
  
 The argument in the proof of Theorem~\ref{th:1} shows that the same conclusion is true in the case where the target surface $F$ is a sphere if $m>3$, and in the case where $F$ is an oriented surface of finite genus when $m=3$. Indeed, in both cases we may choose a trivialization of the tangent bundle of $F\setminus\{p\}$. Therefore, by the argument in the previous paragraph, if $f(\Sigma)$ does not slide through $p$ under the homotopy from $f$ to $g$, the parities of the cumulative winding numbers for $f$ and $g$ are the same. On the other hand, when $f(\Sigma)$ slides through $p$, the cumulative winding number changes by an even number $\pm \chi(F)$, by Lemma~\ref{l:sliding10.1}. Therefore, the parities of the cumulative winding numbers for $f$ and $g$ are the same. Thus, the parities of $\#|\Sigma(f)|$ and $\#|\Sigma(g)|$ are the same. 
  \end{proof}
 
 \begin{remark} We do not know if the statement of Theorem~\ref{th:2} is true for arbitrary closed oriented surfaces $F$ when $m>3$. 
 \end{remark}
 
\section{The invariant $I$ and proof of Theorem ~\ref{th:3}}

In this section we prove Theorem ~\ref{th:3}. The main ingredient of the proof is the $\Z_4$-valued homotopy invariant $I(f)$ defined in the introduction. We will recall the precise definition of the function $I(f)$ in the statement of Lemma~\ref{lem:inv}.

Let $M$ be a closed manifold of dimension $m\ge 2$, and $f\co M\to F$ a smooth stable map to a surface $F$. Then, the singular set $\Sigma(f)$ is a closed $1$-dimensional submanifold of $M$,  which consists of fold points $A_1(f)$, and finitely many cusp points $A_2(f)$.  Recall, the number of components of the singular set $\Sigma(f)$ is denoted by $\#|\Sigma(f)|$, while the number of cusp points is denoted by $\#|A_2(f)|$. We will also consider the number of self-intersection points $\Delta(f)$ of $f(\Sigma)$. We note that if $f$ is generic, then the image of cusp points is not at the self-intersection points of $f(\Sigma)$.

\begin{lemma}\label{lem:inv} Let $f, g\co M\to F$ be two generic maps of a closed manifold of dimension $m\ge 2$ into a surface. Suppose that there exists a generic homotopy $H\co M\times [0,1]\to F\times [0,1]$ between $f$ and $g$ such that the singular set $\Sigma(H)$ is an orientable submanifold of $M\times [0,1]$. Then $I(f)=I(g)$ where 
\[
     I \equiv \#|A_2| + 2\Delta +  2\#|\Sigma| \ (\mod 4).
\]
\end{lemma}

\begin{proof}  Let $H\co M\times [0,1]\to N\times [0,1]$ be a generic homotopy such that $H(x,0) = f(x)$ and $H(x,1) = g(x)$.  Under the homotopy $H$,  the singular set of $f$ may be modified by any of the six allowable homotopy moves.  Let $s$ denote the number of swallowtail moves and their inverses,  and $m$ the number of wrinkles, cusp merges,  and the inverses of these  moves. $R_3$ moves do not change the number of self-intersection points.  Under $R_2$ and cusp-fold moves the number of self-intersection points may change, but the congruence class of $2\Delta(f)$ does not change modulo $4$. Therefore, 
\[
2\Delta(g)\equiv 2\Delta(f)+2s ~(\mod 4),
\]
since every swallowtail move and their inverse changes the number of self-intersection points of the image of the singular set by $1$.  On the other hand, we have 
\[
\#|A_2(g)| \equiv \#|A_2(f)|+2s+2m  ~(\mod 4),
\]
since every swallowtail move, wrinkle, cusp merge and their inverse changes the number of cusps by two.  Next, since the singular set of the homotopy $H$ is orientable, every wrinkle, cusp merge and their inverse changes the parity of $\#|\Sigma(f)|$.  Consequently, we also have the congruence 
\[
2\#|\Sigma(g)|\equiv 2\#|\Sigma(f)|+ 2m ~(\mod 4).
\]
To summarize, 
\[
     2\#|\Sigma(g)| + 2\Delta(g) + \#|A_2(g)| \equiv 2\#|\Sigma(f)|+ 2\Delta(f)+ \#|A_2(f)|+4s +4m ~ (\mod 4),
\]
which simplifies to 
\[
     2\#|\Sigma(g)| + 2\Delta(g) + \#|A_2(g)| \equiv 2\#|\Sigma(f)|+  2\Delta(f) + \#|A_2(f)| ~(\mod 4),
\]
yielding 
\[
   I(g) \equiv I(f) ~(\mod 4).
\]
\end{proof}

\begin{remark}\label{rmk:1}
When the manifold $M$ is even dimensional and the surface $F$ is orientable, the singular set $\Sigma(f)$ is necessarily orientable,  by Theorem ~\ref{th:5}.  Thus,  the function $I(f)$ is a homotopy invariant for generic maps $f:M \rightarrow F$ of a closed manifold of even dimension into an orientable surface. 
\end{remark}

\begin{corollary} The function 
\[
I(f)= \#|A_2(f)| +  2\#|\Sigma(f)| ~(\mod 4)
\] 
is a homotopy invariant of image simple maps $f:M \rightarrow F$, where $M$ is an even dimensional closed manifold and $F$ is an orientable surface. 
\end{corollary}

\begin{corollary} \label{cor:11.3}
The function 
\[
I(f)/2=  \Delta(f)  +  \#|\Sigma(f)| ~(\mod 2)
\] 
is a homotopy invariant of simple stable maps $f:M \rightarrow F$, where $M$ is an even dimensional closed manifold and $F$ is an orientable surface.  \iffalse Moreover, if $g$ is a simple stable map obtained from $f$ via generic homotopy, then 
\[
\Delta(f) + \#|\Sigma(f)| \equiv \Delta(g) + \#|\Sigma(g)| ~(\mod 2). 
\]\fi 
\end{corollary}

Theorem~\ref{th:3} essentially follows from the existence of the invariant $I(f)$. 

\begin{reptheorem}{th:3}
\textit {Let $f$ and $g$ be two homotopic image simple fold maps  %with no cusp singular points
	from a closed manifold $M$ of dimension  $m\geq2$ to a surface $F$ of finite genus. 
	%Suppose that $f|_{\Sigma(f)}$ and $g|_{\Sigma(g)}$ are embeddings.  Also, 
	Suppose the surface $\Sigma(H)$ of singular points of the homotopy $H$ between $f$ and $g$ is orientable. Then, the number of components of $\Sigma(f)$ is congruent modulo two to the number of components of $\Sigma(g)$. }
\end{reptheorem}

 \begin{proof} 
 Consider the homotopy invariant 
 \[
 I(f) = \#|A_2(f)| + 2\Delta(f) +  2\#|\Sigma(f)| ~(\mod 4).
 \]
 
   By assumption, the maps $f$ and $g$ have no cusps and are embedded, therefore $\#|A_2(f)| = \Delta(f) =0$ and $\#|A_2(g)| = \Delta(g) =0$.   Therefore,
 \[
I(f) =2\#|\Sigma(f)| ~(\mod 4) ~\textrm{and} ~ I(g) =2\#|\Sigma(g)| ~(\mod 4).
 \]
By Lemma~\ref{lem:inv},  we have $I(f) \equiv I(g)$.  Thus,
 \[
 2\#|\Sigma(f)| \equiv 2\#|\Sigma(g)| \ (\mod 4)
 \]
which results in 
 \[
 \#|\Sigma(f)| \equiv \#|\Sigma(g)| \ (\mod 2).
 \]
\end{proof}

\section{Low dimensional applications}\label{s:e12}

 In this section we consider examples and applications in the cases of maps to surfaces of manifolds of dimension $m=2,3$ and $4$.

\subsection{Maps of Surfaces to Surfaces}
Let $f:F_g \rightarrow F_h$ be an image simple stable map of oriented surfaces of genera $g$ and $h$, respectively.  By Theorem~\ref{th:1}, the number of path components in $\Sigma(f)$ depends only on the homotopy class of $f$. In fact, a stronger statement is true. 

\begin{proposition}[M.Yamamoto~\cite{Ya}]
\[
\#|\Sigma(f)| \equiv \mathrm{deg}(f)(h-1) - (g-1) ~(\mod 2).
\]
\end{proposition}

The above proposition holds for arbitrary $g,h \geq 0$ and even for non-embedded singular value sets of fold maps. For example, for every fold map $f$ of a sphere into itself (possibly with self-intersecting fold curve $f(\Sigma)$), we have 
\[
   \#|\Sigma(f)| \equiv \mathrm{deg}(f) - 1 ~(\mod 2).
\]

\subsection{Maps of the $3$-sphere to the $2$-sphere} 

In ~\cite{Sae},  Saeki studied fold maps of $3$-dimensional manifolds into surfaces and showed that every closed connected oriented $3$-manifold admits a stable map to the $2$-sphere without definite fold points. In particular,  for maps of the $3$-sphere into the $2$-sphere, Saeki constructed an image simple indefinite fold map $f:S^3 \rightarrow S^2$ such that $\Sigma(f) = n+1$, where $n \in \Z$ is the Hopf invariant $H(f)$ of $f$. Saeki posed the following question. 

\begin{problem}\label{p:2} For an integer $n\in \Z\simeq \pi_3S^2$, let us consider stable maps $f\co S^3\to S^2$ without definite fold which represent the associated homotopy class and which satisfies that $\Sigma(f)\ne \emptyset$ and $f|_{\Sigma(f)}$ is an embedding, where $\Sigma(f)$ is the set of singular points of $f$. Then, is the number of components of $\Sigma(f)$ congruent modulo two to $n+1$?
\end{problem}
 Saeki solved Problem~\ref{p:2} in the negative in ~\cite{Sae19}. The following corollary shows under what conditions Saeki's problem can be answered in the positive. 
\noindent As a corollary of Theorems ~\ref{th:2} and ~\ref{th:3}, we prove the following statement related to Saeki's question.

\begin{corollary}\label{cor:1a}
Let $f:S^3 \rightarrow S^2$ be an image simple indefinite fold map with Hopf invariant $H(f) = n$ constructed by Saeki in \cite{Sae}. If $g:S^3 \rightarrow S^2$ is obtained from $f$ by a homotopy \[ F:S^3 \times [0,1] \rightarrow S^2 \times [0,1] \]   such that $\Sigma(F)$ is orientable or $F(\Sigma)$ has no triple self-intersection points,
then  
\[
\#|\Sigma(g)| \equiv \#|\Sigma(f)| \equiv n+1 ~(\mod 2).
\]
\end{corollary}

\subsection{Maps of the $4$-sphere to the $2$-sphere}As a consequence of Theorem ~\ref{th:1}, we obtain a result on the $4$-dimensional analog of Problem ~\ref{p:2}. 

\begin{corollary}\label{cor:1}
Let $f: S^4 \rightarrow S^2$ be an image simple fold map of the $4$-sphere into the $2$-sphere.  Then 
\[
\#|\Sigma(f)| \equiv 1 ~(\mod 2).
\]
\end{corollary}

\begin{proof}
Let us examine an image simple fold map representative of both the trivial and non-trivial elements of $\pi_4(S^2) \cong \Z_2$.  We respectively denote the equivalence classes of the trivial and non-trivial elements of $\pi_4(S^2)$ by $[0]$ and $[1]$.  The trivial element is constructed via the standard projection to $\R^2$,  followed by the inclusion into $S^2$,  i.e.  $f_{[0]}: S^4 \rightarrow \R^2 \hookrightarrow S^2$,  where $f_{[0]}(\Sigma)$ consists of one closed embedded definite fold curve.  Therefore, by Theorem~\ref{th:1},  any image simple fold map $g\in[0]$ has a singular set such that $\#|\Sigma(g)|$ is odd.    

\begin{figure}[h]
\includegraphics[width=0.6\textwidth]{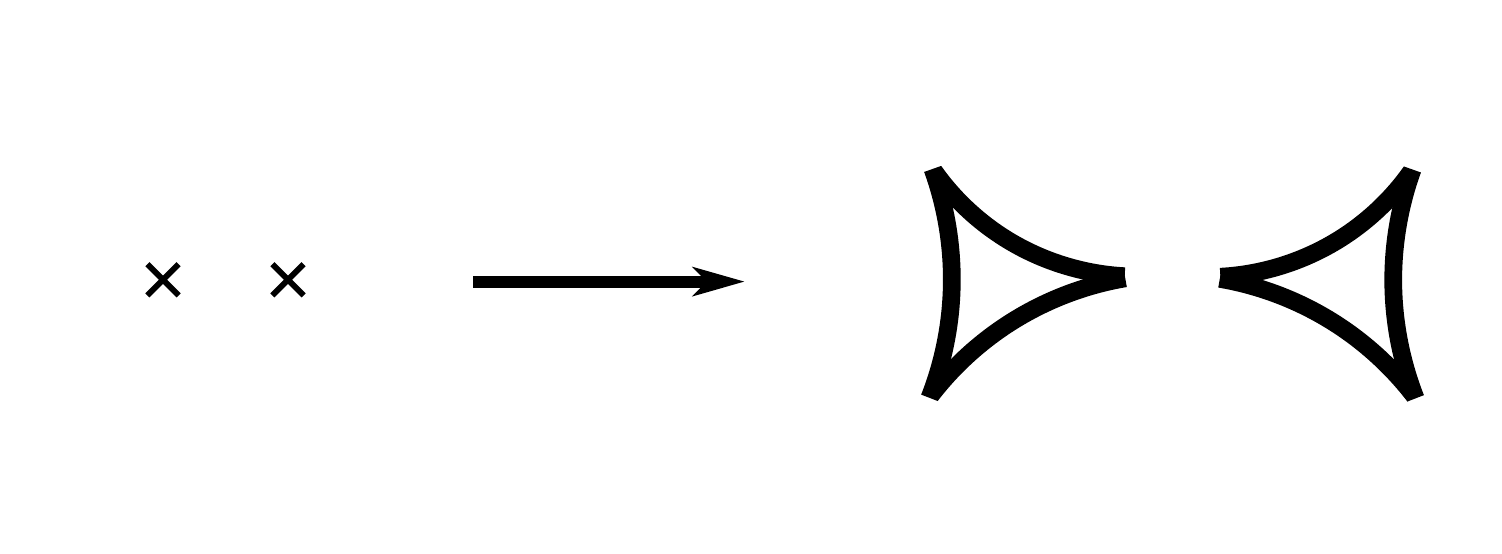}
\caption{Replacing Lefschetz critical points with cusp and indefinite fold points.}
\label{fig:17}
\end{figure}

\begin{figure}[h]
\includegraphics[width=0.6\textwidth]{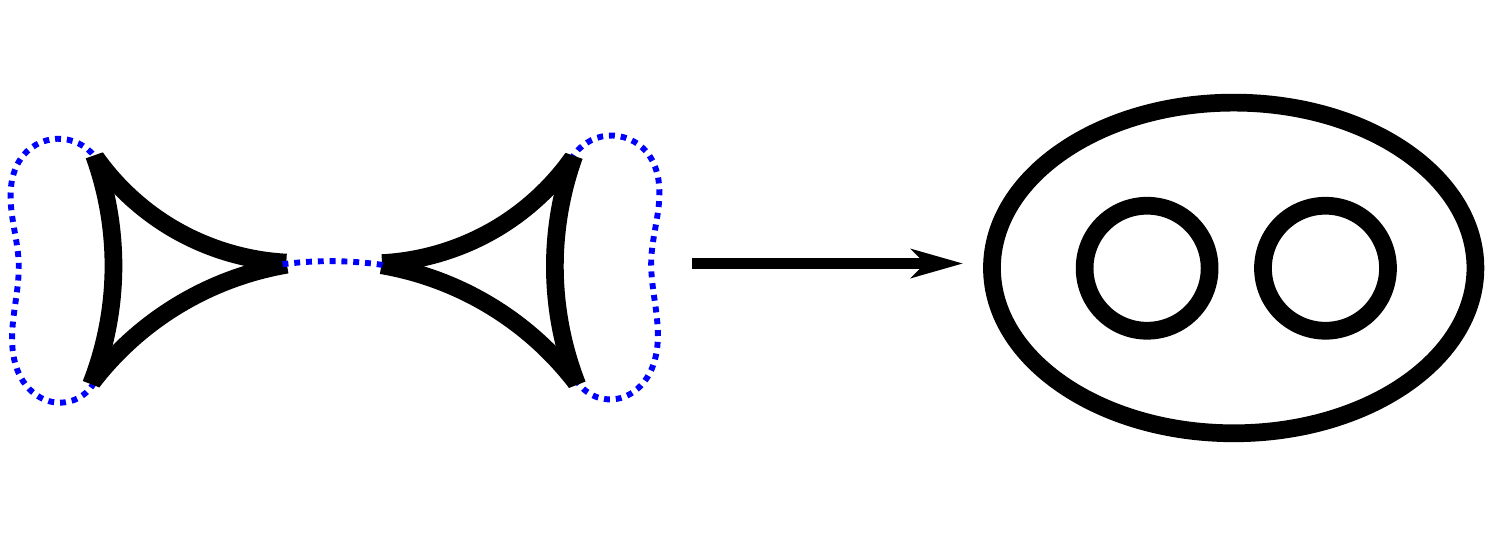}
\caption{Three cusp merges.}
\label{fig:18}
\end{figure}

Next,  we examine the non-trivial element of $\pi_4(S^2)$. Consider the suspension of the Hopf fibration $H:S^3 \rightarrow S^2$,  defined as $\Sigma H: \Sigma S^3 \rightarrow \Sigma S^2$,  which is equivalent to $\Sigma H: S^4 \rightarrow S^3$.  Composition of the suspended Hopf fibration with the Hopf fibration itself results in the map $f_{[1]}: H \circ \Sigma H: S^4 \rightarrow S^2$.  The singular set of $f_{[1]}$ consists of a pair of Lefschetz critical points,  see ~\cite{Ma} for a detailed explanation.  
Each Lefschetz critical point can be deformed into a component consisting of three cusps and indefinite folds as in Figure ~\ref{fig:17}.  For an explicit description of the move in Figure ~\ref{fig:17}, we refer the reader to the third section of \cite{Le}.

We then obtain an embedding of three indefinite fold components after thrice merging pairs of the recently created cusps, see Figure ~\ref{fig:18}.  Now,  the singular set of the image simple fold map $f_{[1]}$ has an odd number of components and thus,  by Theorem ~\ref{th:1},  the singular set of any image simple fold map $h\in[1]$ must also have an odd number of connected components.  

Up to homotopy,  we have examined the singular set of all image simple fold maps from the $4$-sphere to the $2$-sphere.  In all cases,  the singular set has an odd number of connected components. 
\end{proof}

\begin{remark} We note that through steps described in ~\cite{Sae}, every image simple fold map is homotopic to an image simple indefinite fold map. 
\end{remark}

Combining the statement of Remark ~\ref{rmk:1} with Corollary ~\ref{cor:1},  we obtain the following corollary.

\begin{corollary}\label{cor:12.6}
For every smooth stable map $f:S^4 \rightarrow S^2$, we have 
\[
I(f) \equiv 2 \ (\mod 4).
\]
\end{corollary}

We may also combine Corollary ~\ref{cor:11.3} and Corollary ~\ref{cor:12.6} to get the following result. 
\begin{corollary}
For every simple stable map $f:S^4 \rightarrow S^2$,  we have
\[
\Delta(f) \equiv \#|\Sigma(f)| + 1 ~(\mod 2). 
\]
\end{corollary}

\end{document}